\documentclass[a4paper,10pt]{article}
\usepackage{mypackages}
\usepackage{mymacros}
\usepackage{mystyle}
\title{Coarse sheaf cohomology}
\author{Elisa Hartmann\thanks{Department of Mathematics, Karlsruhe Institute of Technology}} 

\begin{document}

\allowdisplaybreaks

\maketitle

\begin{abstract}
A certain Grothendieck topology assigned to a metric space gives rise to a sheaf cohomology theory which sees the coarse structure of the space. Already constant coefficients produce interesting cohomology groups. In degree 0 they see the number of ends of the space. In this paper a resolution of the constant sheaf via cochains is developed. It serves to be a valuable tool for computing cohomology. In addition coarse homotopy invariance of coarse cohomology with constant coefficients is established. This property can be used to compute cohomology of Riemannian manifolds. The Higson corona of a proper metric space is shown to reflect sheaves and sheaf cohomology. Thus we can use topological tools on compact Hausdorff spaces in our computations. In particular if the asymptotic dimension of a proper metric space is finite then higher cohomology groups vanish. We compute a few examples. As it turns out finite abelian groups are best suited as coefficients on finitely generated groups.
\end{abstract}

\tableofcontents

\section{Introduction}

The sheaf-theoretic approach to coarse metric spaces has been applied in many different contexts \cite{Bunke2017,Roe2013,Schmidt1999}. Sheaf-theoretic methods play an important role in our paper. We also present three other computational tools. Cochain complexes assigned to a filtration of Vietoris-Rips complexes have not just been used in the coarse setting \cite{Hausmann1995}. Many well-known coarse (co-)homology theories are coarse homotopy invariant \cite{Higson1994,Mitchener2001,Wulff2020}. The cohomology of the Higson corona is of course as a composition of functors a coarse invariant which has been studied before \cite{Keesling1994}. Even in combination with other computational methods \cite{Higson1993} coarse invariants are hard to compute for the spaces one is most interested in which include Riemannian manifolds and finitely generated groups. 

Coarse sheaf cohomology has been designed by the author in her thesis. Aside from an agenda to present new computational methods which may be suitable for a large number of spaces there are two immediate results:

\begin{thm}
 If $M$ is a non-positively curved closed Riemannian $n$-manifold and $A$ a finite abelian group then 
 \[
  \check H_{ct}^q(M,A)=H_{sing}^q(S^{n-1};A)
 \]
the left side denotes coarse sheaf cohomology with values in the constant sheaf $A$ and the right side denotes singular cohomology with values in the group $A$.
\end{thm}

This result can be immediately applied to define a coarse version of mapping degree associated to a coarse map between manifolds.

\begin{thm}
 If $T$ is a simplicial tree with infinitely many ends and $A$ is a finite abelian group then
 \[
  \check H_{ct}^q(T;A)=\begin{cases}
                        \bigoplus_\N A & q=0\\
                        0 & \mbox{otherwise.}
                       \end{cases}
 \]
\end{thm}

There are many interesting cohomology theories on coarse metric spaces. The most prominent examples are Roe's coarse cohomology \cite{Roe1993,Roe2003,Hair2010} and controlled operator $K$-theory \cite{Roe1996,Higson2000,Yu1995,Yu2000}. If two metric spaces $X,Y$ have the same coarse type then specifying a coarse equivalence $X\to Y$ is a proof. If on the other hand $X,Y$ do not have the same coarse type then a coarse invariant which does not have the same values on $X$ and $Y$ gives a proof. In general a well designed cohomology theory delivers a rich source of invariants which are easy to compute. To this date cohomology of finitely generated free abelian groups has been calculted for Roe's coarse cohomology and also controlled operator $K$-theory. There is still a gap in knowledge about cohomology of other finitely generated groups. Riemannian manifolds on the other hand do not show interesting cohomology groups since every Riemannian $n$-manifold with nonpositive sectional curvature is coarsely homotopic to $\R^n$ and most coarse cohomology theories are coarse homotopy invariant \cite{Roe1993}.

Our coarse cohomology theory $\check H^q_{ct}(\cdot;\cdot)$ is a sheaf cohomology theory on a Grothendieck topology $X_{ct}$ assigned to a metric space $X$ \cite{Hartmann2017a}. If $A$ is an abelian group then for the constant sheaf $A$ on $X$ we obtain in dimension $0$ a copy of $A$ for every end of $X$ or an infinite direct sum of copies of $A$ if $X$ does not have finitely many ends \cite{Hartmann2017a}.

In this paper we design a cochain complex $(CY^q_b(X;A))_q$ assigned to a metric space $X$ and abelian group $A$. The functor $U\s X\mapsto CY^q_b(U,A)$ forms a flabby sheaf on $X_{ct}$. The sequence of sheaves
\[
 0\to A_X\to CY^0_b(\cdot,A)\to CY^1_b(\cdot,A)\to CY^2_b(\cdot,A)\to\cdots
\]
is exact. Thus cohomology of
\[
 0\to CY^0_b(X,A)\to CY^1_b(X,A)\to CY^2_b(X,A)\to \cdots
\]
computes coarse sheaf cohomology of $X$ with values in $A_X$.

\begin{thma}
\label{thm:introCochains}
 If $X$ is a metric space then there is a flabby resolution $CY^q_b(\cdot;A)$ of the constant sheaf $A_X$ on $X_{ct}$. We can compute sheaf cohomology with values in the constant sheaf using cochain complexes:
 \[
  \check H_{ct}^q(X;A_X)=HY^q_b(X,A).
 \]
\end{thma}

For $q\ge 1$ there is a comparison map $HY_b^q(X;A)\to HX^{q+1}(X;A)$ with Roe coarse cohomology. This map is neither injective nor surjective though. The main difference is that our cochains are defined as maps that need to be ``blocky'' while coarse cochains do not have this restriction. Thus general statements on cohomology are easier to prove for Roe coarse cohomology. While we hope that combinatorical computations are easier realized using blocky cochains.

There are several notions of homotopy on the coarse category which are all equivalent in some way. The homotopy theory we are going to employ uses the asymptotic product as coarse substitute for a product and the first quadrant in $\R^2$ equipped with the Manhattan metric as a coarse substitute for an interval \cite{Hartmann2019b}. In effect this homotopy theory and the other coarse homotopy theories are only of use if one wants to compute cohomology of $\R^n$ and maybe Riemannian manifolds. Nonetheless we prove that coarse sheaf cohomology is a coarse homotopy invariant using the resolution via cochains.

\begin{thma}
\label{thm:introHomotopy}
 If two coarse maps $\alpha,\beta:X\to Y$ between metric spaces are coarsely homotopic then they induce the same map
 \[
  \alpha^*,\beta^*:\check H_{ct}^q(Y;A)\to \check H_{ct}^q(X;A)
 \]
 in cohomology with values in a constant sheaf $A$.
\end{thma}

A coarse map $\alpha:X\to Y$ between metric spaces induces a cochain map $\alpha^*:CY_b^q(Y;A)\to CY_b^q(X;A)$ which in turn induces a homomorphism $\alpha^*:HY_b^q(Y;A)\to HY_b^q(X;A)$. Conversely the inverse image functor maps the constant sheaf $A_Y$ on $Y_{ct}$ to the constant sheaf $A_X$ on $X_{ct}$. Thus there is an induced homomorphism $\alpha^*$ in cohomology. One may wonder if both homomorphisms $\alpha^*,\alpha^*$ coincide. And indeed they do.

To a proper metric space $X$ we can assign a compact Hausdorff topological space $\nu(X)$, the Higson corona of $X$. This version of boundary reflects sheaf cohomology in the following way: There is a functor $\cdot ^\nu$ which maps a sheaf $\mathcal F$ on $X_{ct}$ to a sheaf $\mathcal F^\nu$ on $\nu(X)$. Conversely the functor $\hat\cdot$ maps a sheaf $\mathcal G$ on $\nu(X)$ to a sheaf $\hat{\mathcal G}$ on $X_{ct}$, Together they provide an equivalence of categories between ``reflective'' sheaves on $X_{ct}$ and sheaves on $\nu(X)$. In particular the constant sheaf $A_X$ on $X_{ct}$ is reflective and mapped to the constant sheaf $A_{\nu(X)}$ on $\nu(X)$. We can compute cohomology with constant coefficients either way:

\begin{thma}
\label{thm:introHigson}
 If $X$ is a proper metric space and $A$ an abelian group then
 \[
  \check H^q_{ct}(X;A_X)=\check H(\nu(X);A_{\nu(X)})
 \]
the $q$th cohomology of $X$ with values in the constant sheaf $A$ on $X_{ct}$ is isomorphic to the $q$th sheaf cohomology of the Higson corona $\nu(X)$ of $X$ with values in the constant sheaf $A_{\nu(X)}$.

Moreover if $\asdim(X)\le n$ then $\check H_{ct}^q(X,A_X)=0$ for $q>n$.
\end{thma}

This paper provides enough computational methods to compute metric cohomology of finitely generated groups. Vanishing of $\check H_{ct}^1(\Z,A)=0$ for finite $A$ can be computed directly using cochains. Then our result on the Higson corona implies that $\Z$ is acyclic for finite coefficients. The same method can be employed to show that trees are acyclic for finite coefficients. Thus we computed metric cohomology of the free group $F_n$ with $n<\infty$ generators. Computing cohomology of the free abelian groups $\Z^n$ with $n<\infty$ is more challenging. A coarse homotopy equivalence $\Z^{n-1}\times \Z_{\ge0}\to \Z_{\ge0}$ provides a Leray cover of $\Z^n$ which has the same combinatorical information as the nerve of a Leray cover of the topological space $S^{n-1}$. Thus cohomology with finite coefficients can be derived.

\begin{thma}
\label{thm:introZn}
 If $A$ is a finite abelian group then
 \[
  \check H^q_{ct}(\Z^n;A)=\begin{cases}
                      A\oplus A & n=1,q=0\\
                      A &n\not=1,q=0\vee q=n-1\\
                      0 &\mbox{otherwise}.
                     \end{cases}
 \]
\end{thma}

There is a more general notion of coarse space which includes the class of coarse metric spaces. Most of our concepts work in more generality. We restrict our attention to metric spaces only since a wider audience (than coarse geometers) is interested in this class of coarse spaces only. The coarse sheaf cohomology theory is defined on coarse spaces with connected coarse structure. The resolution via cochains also works for this class of spaces. The homotopy theory is only defined for metric spaces and the results on the Higson corona work for proper metric spaces and coarse structures generated by a compactification of a paracompact, locally compact Hausdorff space.

This article is organized in 10 chapters. Some can be read independently but there are also a few dependencies as depicted in the following diagram.
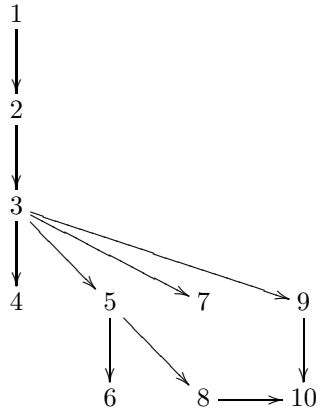
\[
 \xymatrix{
  1\ar[d]& & & \\
  2\ar[d] & & & \\
   3\ar[d]\ar[dr]\ar[drr]\ar[drrr]&&&\\
   4
   &5\ar[d]\ar[dr]
   &7
   &9\ar[d]\\
   &6
   &8\ar[r]
   &10
 }
\]

The final chapter uses every aspect so far discussed. Sheaf-theoretic methods, the resolution via cochains, coarse homotopy and the Higson corona are employed in the computation of metric cohomology of $\Z^n$.

\section{Coarse cohomology by Roe and the Higson corona}
This chapter introduces terminology and concepts which are well known to coarse geometers.

If $X$ is a metric space then a subset $E\s X\times X$ is called an \emph{entourage} if
\[
 \sup_{(x,y)\in E}d(x,y)<\infty
\]
The set of entourages forms the \emph{coarse structure} of $X$. If $R\ge 0$ then the set 
\[
 \Delta_R:=\{(x,y)\in X\times X|d(x,y)\le R\}
\]
is an entourage. If $E\s X\times X,B\s X$ are two subsets then
\[
 E[B]:=\{x\in X\mid (x,y)\in E, y\in B\}
\]
A subset $B\s X$ is called \emph{bounded} if there exists $x_0\in X$ and $R\ge 0$ such that $\Delta_R[x_0]\z B$.

A map $\alpha:X\to Y$ between metric spaces is called \emph{coarsely uniform} if for every $R\ge 0$ there exists $S\ge 0$ such that $d(x,y)\le R$ in $X$ implies $d(\alpha(x),\alpha(y))\le S$ in $Y$. The map $\alpha$ is called \emph{coarsely proper} if for every bounded set $B\s Y$ the set $\alpha^{-1}(B)$ is bounded in $X$. The map $\alpha$ is called \emph{coarse} if $\alpha$ is both coarsely uniform and coarsely proper. Two maps $\alpha,\beta:X\to Y$ between metric spaces are called \emph{close} if the set $\alpha\times\beta(\Delta_0)$ is an entourage in $Y$. The \emph{coarse category} consists of metric spaces as objects and coarse maps modulo close as morphisms. Isomorphisms in this category are called \emph{coarse equivalences}.

This paper presents a resolution of the constant sheaf which consists of cochains which closely resemble coarse cochains of Roe's coarse cohomology. For this purpose we give a quick  introduction to coarse cohomology by Roe which was invented by Roe in \cite{Roe1993,Roe2003}.

If $X$ is a metric space then the \emph{set of $q$-simplices of the $R$-Vietoris-Rips complex of $X$} is defined as
\[
 \Delta^q_R:=\{(x_0,\ldots,x_q)\mid d(x_i,x_j)\le R\forall i,j\}.
\]
A subset $B\s X^{q+1}$ is called \emph{bounded} if the projection to every factor is bounded. Then a subset $C\s X^{q+1}$ is called \emph{cocontrolled} if for every $R\ge 0$ the set $C\cap \Delta^q_R$ is bounded.

\begin{defn}
 If $X$ is a metric space and $A$ an abelian group then the \emph{coarse cochains} $CX^q(X;A)$ is the set of functions $X^{q+1}\to A$ with cocontrolled support. It is a group by pointwise addition. The coboundary map $\partial_q:CX^q(X;A)\to CX^{q+1}(X;A)$ is defined by
 \[
  (\partial_q\varphi)(x_0,\ldots,x_{q+1})=\sum_{i=0}^{q+1}(-1)^i\varphi(x_0,\ldots,\hat x_i,\ldots,x_{q+1}).
 \]
This makes $(CX^q(X;A),\partial_q)$ a cochain complex. Its homology is called \emph{coarse cohomology by Roe} and denoted by $HX^*(X;A)$.
\end{defn}

If $\alpha:X\to Y$ is a coarse map then it induces a cochain map 
\begin{align*}
 \alpha^q:CX^q(Y,A)&\to CX^q(X,A)\\
 \varphi&\mapsto \varphi\circ\alpha^{\times(q+1)}
\end{align*}
Two coarse maps which are close induce the same map in cohomology. If $A=\R$ and $X=\R^n$ then
\[
 HX^q(\R^n,\R)=\begin{cases}
                \R & q=n\\
                0 &\mbox{otherwise}.
               \end{cases}
\]

A section in this paper transfers sheaves on a proper metric space to sheaves on its Higson corona. For this purpose we give a definition of the Higson corona which is equivalent to the usual one \cite{Hartmann2019c}. 

Let $X$ be a metric space. Two subsets $A,B\subseteq X$ are called \emph{close} (or \emph{not coarsely disjoint}) if there exists an unbounded sequence $(a_i,b_i)_i\subseteq A\times B$ and some $R\ge 0$ such that $d(a_i,b_i)\le R$ for every $i$. We write $A\curlywedge B$ in this case.

Let $R>0$ be a real number. A metric space $X$ is called \emph{$R$-discrete} if $d(x,y)\ge R$ for every $x\not=y$. If $X$ is a metric space an $R$-discrete for some $R>0$ subspace $S\subseteq X$ is called a \emph{Delone set} if the inclusion $S\to X$ is coarsely surjective. Every metric space contains a Delone set. If it is in addition proper then the finite sets of the Delone set are exactly the bounded sets.

\begin{defn}
\label{defn:Higson corona}
Let $X$ be a proper metric space and $S\subseteq X$ a Delone subset. Denote by $\hat S$ the set of nonprincipal ultrafilters on $S$. If $A\subseteq S$ is a subset define
\[
\closedop A:=\{\mathcal F\in \hat S:A\in \mathcal F\}.
\]
Then define a relation $\curlywedge$ on subsets of $\hat S$: $\pi_1 \curlywedge \pi_2$ if for every $A,B\subseteq S$ the relations $\pi_1\subseteq \closedop A,\pi_2\subseteq \closedop B$ imply $A \curlywedge B$.

The relation $\curlywedge$ on subsets of $\hat S$ determines a Kuratowski closure operator
\[
\bar \pi=\{\mathcal F\in \hat S:\{\mathcal F\} \curlywedge \pi\}.
\]
Now define a relation $\lambda$ on $\hat S$: $\mathcal F \lambda \mathcal G$ if 
$A\in \mathcal F,B\in \mathcal G$ implies $A \curlywedge B$. 

 Now the \emph{Higson corona $\nu(X)$} of $X$ is defined $\nu(X)=\hat S/\lambda$ as the quotient by $\lambda$.
\end{defn}

 If $A\subseteq X$ is a subset of a metric space then $\closedop A=\bar A\cap \nu(X)$ where the closure is taken in the Higson compactification. We call $(\closedop A^c)_{A\subseteq X}$ the basic open sets and $(\closedop A)_{A\subseteq X}$ the basic closed sets. There are two observations: If $A,B\subseteq X$ are two subsets then
 \begin{itemize}
  \item $\closedop A\cap \closedop B=\emptyset$ if and only if $A\not\curlywedge B$ are not close
  \item $\closedop A\cup \closedop B=\closedop{A\cup B}$.
 \end{itemize}

\section{Coarse sheaf cohomology, a survey}
This chapter gives a survey on coarse sheaf cohomology or coarse cohomology with twisted coefficients as we call it \cite{Hartmann2017a}. There are several facts on sheaf cohomology on topological spaces which hold in more generality for sheaf cohomology defined on a Grothendieck topology. Since the literature does not provide every aspect we are going to prove these facts by hand.

Let $U\s X$ be a subset of a metric space. A finite family of subsets $U_1,\ldots,U_n\s U$ forms a \emph{coarse cover} of $U$ if for every entourage $E\s U\times U$ the set
\[
 E[U_1^c]\cap\cdots\cap E[U_n^c]
\]
is bounded. This is equivalent to saying that the set
\[
 (U\times U)\cap(\bigcup_i U_i\times U_i)^c
\]
is a cocontrolled subset of $X^2$.

To a metric space $X$ we associate a Grothendieck topology $X_{ct}$ in the following way. The underlying category $Cat(X_{ct})$ is the poset of subsets of $X$. Subsets $(U_i)_i$ form a covering of $U\s X$ if they coarsely cover $U$.

A contravariant functor $\sheaff$ on subsets of $X$ is a \emph{sheaf on $X_{ct}$} if for every coarse cover $U_1,\ldots,U_n\s U$ of a subset of $X$ the following diagram is an equalizer
\[
 \sheaff(U)\to\prod_i\sheaff(U_i)\rightrightarrows\prod_{i,j
 }\sheaff(U_i\cap U_j).
\]
If $\sheaff$ is a sheaf on $X_{ct}$ then the right derived functor of the global sections functor is called \emph{coarse sheaf cohomology}, written $\check H_{ct}^*(X;\sheaff)$.

If 
\[
0\to \mathcal F \to \mathcal G\to \mathcal H\to 0
\]
is a short exact sequence of sheaves on $X_{ct}$ then there is a long exact sequence in cohomology
\begin{align*}
 0&\to \check H_{ct}^0(X,\mathcal F)\to \check H_{ct}^0(X,\mathcal G)\to \check H_{ct}^0(X,\mathcal H)\\
 &\to \check H_{ct}^1(X,\mathcal F)\to\cdots\\
 \cdots&\to \check H_{ct}^q(X,\mathcal F)\to \check H_{ct}^q(X,\mathcal G)\to \check H_{ct}^q(X,\mathcal H)\\
 &\to \check H_{ct}^{q+1}(X,\mathcal F)\to\cdots
\end{align*}

If $A$ is an abelian group the sheafification of the constant presheaf $A$ on $X_{ct}$ is called the \emph{constant sheaf $A_X$ on $X$}. In this paper we are interested in the computation of $\check H_{ct}^q(X;A_X)$ in higher dimension. The zeroth cohomology group is related to the number of ends $e(X)$ of the metric space $X$:
\[
 \check H_{ct}^0(X;A_X)=A_X(X)=\begin{cases}
                            A^{e(X)} & e(X)<\infty\\
                            \bigoplus_{\N}A & e(X)=\infty.
                           \end{cases}
\]

A sheaf $\mathcal F$ on $X_{ct}$ is called \emph{acyclic} if $\check H_{ct}^q(X,\mathcal F)=0$ for $q>0$. A sequence of sheaves 
\[
 \cdots \to \mathcal F_i\xrightarrow{\varphi_i} \mathcal F_{i+1}\xrightarrow{\varphi_{i+1}} \mathcal F_{i+2}\to \cdots
\]
is exact if $\im \varphi_i=\ker \varphi_{i+1}$. Here $\im \varphi$ is the sheafification of the image presheaf $U\mapsto \im(\varphi(U))$.

\begin{lem}
 If 
 \[
  0\to \mathcal F\to \mathcal F_0\to \mathcal F_1\to \cdots
 \]
 is an acyclic resolution of sheaves on $X_{ct}$ then
 \[
  \check H^q_{ct}(X,\mathcal F)=H^q(0\to \mathcal F_0(X)\to \mathcal F_1(X)\to\cdots), 
 \]
 cohomology of $\mathcal F$ can be computed by taking homology of the cocomplex
 \[
  0\to \mathcal F_0(X)\to \mathcal F_1(X)\to \cdots.
 \]
\end{lem}
\begin{proof}
 Suppose
 \[
   0\to \mathcal F\xrightarrow{i} \mathcal F_0\xrightarrow{d_0} \mathcal F_1\xrightarrow{d_1}\mathcal F_2\xrightarrow{d_2} \cdots
 \]
is an exact sequence with $\mathcal F_i$ acyclic for every $i$.

For every $i=1,2,\ldots$ define $\varepsilon_i=\ker d_i$. 
The exact sequence
\[
 0\to\mathcal F\to \mathcal F_0\to \varepsilon_1\to 0
\]
gives rise to a long exact sequence
\begin{alignat*}{2}
0&\to \check H_{ct}^0(X,\mathcal F)\to \check H_{ct}^0(X,\mathcal F_0)&\to \check H_{ct}^0(X,\mathcal \varepsilon_1)\\
 &\to \check H_{ct}^1(X,\mathcal F)\to 0&\to\cdots\qquad\quad\\
 &&0\to \check H_{ct}^q(X,\varepsilon_1)\\
 &\to \check H_{ct}^{q+1}(X,\mathcal F)\to0
 \end{alignat*}
 Thus $\check H_{ct}^{q+1}(X,\mathcal F)=\check H^q(X,\varepsilon_1)$ for $q\ge 1$ and 
 \begin{align*}
 \check H^1_{ct}(X,\mathcal F)
 &=\varepsilon_1(X)/\im (d_0(X))\\
 &=(\ker d_1)(X)/\im(d_0(X))\\
 &=H^1(0\to \sheaff_0(X)\to \sheaff_1(X)\to\cdots)
 \end{align*}
 The inclusion $\varepsilon_i\to \mathcal F_i$ and the corestriction of $d_i$ to $\im d_i=\varepsilon_{i+1}$ combine to an exact sequence
\[
 0\to \varepsilon_i\to \mathcal F_i\to \varepsilon_{i+1}\to 0.
\]
This sequence gives rise to a long exact sequence
\begin{alignat*}{2}
0&\to \check H_{ct}^0(X,\varepsilon_i)\to \check H_{ct}^0(X,\mathcal F_i)&\to \check H_{ct}^0(X,\mathcal \varepsilon_{i+1})\\
 &\to \check H_{ct}^1(X,\varepsilon_i)\to\cdots&\\
 &&0\to \check H_{ct}^q(X,\varepsilon_{i+1})\\
 &\to \check H_{ct}^{q+1}(X,\varepsilon_i)\to0
 \end{alignat*}
which reads $\check H_{ct}^{q+1}(X,\varepsilon_i)=\check H^q_{ct}(X,\varepsilon_{i+1})$ for $q\ge 1$.
If $q\ge 2$ then we obtain inductively
\begin{align*}
 \check H_{ct}^q(X,\sheaff)&=\check H_{ct}^{q-1}(X,\varepsilon_1)=\cdots=\check H_{ct}^1(X,\varepsilon_{q-1})\\
 &=H^1(0\to \sheaff_{q-1}(X)\to \sheaff_q(X)\to\cdots)=H^q(0\to \sheaff_0(X)\to\sheaff_1(X)\to \cdots).
\end{align*}
\end{proof}

\begin{lem}
\label{lem:injective->flabby}
 If $X$ is a metric space every injective sheaf on $X_{ct}$ is flabby.
\end{lem}
\begin{proof}
 Let $\mathcal I$ be an injective sheaf on $X_{ct}$ and let $V\s U$ be an inclusion of subsets. We define a presheaf $\Z_{U,X}$ on $X_{ct}$ by 
 \[
  W\mapsto\begin{cases}
               \Z & W\s U\\
               0 &\mbox{otherwise}
              \end{cases}
 \]
 Denote by $\Z_{U,X}^\#$ the sheafification. In a similar way we define $\Z_{V,X}$ and $\Z_{V,X}^\#$. Then $\Z_{V,X}^\#\le \Z_{U,X}^\#$ is a subsheaf in a canonical way. Thus we have an exact sequence
\[
 0\to \Z_{V,X}^\#\to \Z_{U,X}^\#
\]
Since $\mathcal I$ is an injective object the sequence
\[
 Hom_{Sh}(\Z_{U,X}^\#,\mathcal I)\to Hom_{Sh}(\Z_{V,X}^\#,\mathcal I)\to 0
\]
is exact. Now 
\[
Hom_{Sh}(\Z_{U,X}^\#,\mathcal I)=Hom_{PSh}(\Z_{U,X},\mathcal I)=\mathcal I(U) 
\]
and
\[
Hom_{Sh}(\Z_{V,X}^\#,\mathcal I)=Hom_{PSh}(\Z_{V,X},\mathcal I)=\mathcal I(V)
\]
which proves the claim.
\end{proof}

\begin{lem}
 If $X$ is a metric space then flabby sheaves on $X_{ct}$ are acyclic.
\end{lem}
\begin{proof}
 We mimic the proof of \cite[Proposition~2.5]{Hartshorne1977}.
 
 If $\mathcal F$ is a flabby sheaf then it can be embedded in an injective sheaf $\mathcal I$. The quotient of this inclusion is denoted $\mathcal G$. Then we have an exact sequence
 \begin{align}
 \label{eq:exact1}
  0\to \mathcal F\to \mathcal I\to \mathcal G\to 0
 \end{align}
with $\sheaff$ flabby, $\mathcal I$ flabby by Lemma~\ref{lem:injective->flabby} and $\sheafg$ is flabby by a standard argument. General theory on flabby sheaves also implies that the sequence 
\begin{align}
 \label{eq:exactX}
 0\to \sheaff(X)\to \mathcal I(X)\to \mathcal G(X)\to 0
\end{align}
is exact. Then the long exact sequence in cohomology to the short exact sequence~\ref{eq:exact1}, the exactness of~\ref{eq:exactX} and $\check H_{ct}^q(X,I)=0$ for $q\ge 1$ implies $\check H_{ct}^1(X,\sheaff)=0$ and
\begin{align}
 \label{eq:shift}
 \check H_{ct}^q(X,\sheaff)=\check H_{ct}^{q-1}(X,\sheafg)
\end{align}
for $q\ge 2$. Since $\mathcal G$ satisfies the requirements for this Lemma we obtain the result for $q\ge 2$ using inductively $\mathcal G$ and the isomorphism~\ref{eq:shift}.
\end{proof}

\section{Standard resolution}
This chapter proves Theorem~\ref{thm:introCochains}.

Let $A$ be an abelian group. If $A_1\sqcup\cdots\sqcup A_n=U$ is a disjoint union of a subset $U\s X$ of a metric space then
\[
 C^q_{A_1,\ldots,A_n}(U,A)=\{\varphi:U^{q+1}\to A\mid\varphi|_{A_{i_0}\times\cdots \times A_{i_q}}\mbox{ constant } \forall i_0,\ldots,i_q\in\{1,\ldots,q\}\}
\]
 Then we define 
 \[
 C^q(U,A)=\varinjlim_{A_1\sqcup\cdots \sqcup A_n=U}C^q_{A_1,\ldots,A_n}(U,A)
 \]
 where $\varphi\in C^q_{A_1,\ldots,A_n}(U,A)$ is equivalent to $\psi\in C^q_{B_1,\ldots,B_m}(U,A)$ if
 \[
  \varphi|_{A_{i_0}\cap B_{j_0}\times\cdots\times A_{i_q}\cap B_{j_q}}=\psi|_{A_{i_0}\cap B_{j_0}\times\cdots\times A_{i_q}\cap B_{j_q}}
 \]
for every $i_0,\ldots i_q\in \{1,\ldots,n\},j_0,\ldots,j_q\in \{1,\ldots,m\}$. We equip $C^q(U;A)$ with a group operation by pointwise addition. The elements of $C^q(U;A)$ are called \emph{blocky} functions with blocks $A_1,\ldots, A_n$. They can be compared with the group of \emph{all} functions $U^{q+1}\to A$.

A differential on $C^q(U,A)$ is defined by
\begin{align*}
 d_q:C^q(U,A)&\to \mathcal C^{q+1}(U,A)\\
 \varphi&\mapsto((x_0,\ldots,x_{q+1})\mapsto\sum_{i=0}^{q+1}(-1)^i\varphi(x_0,\ldots,\hat x_i,\ldots,x_{q+1}).
\end{align*}

Now $CX^q_b(U,A)$ defines the subcomplex of functions in $C^q(U,A)$ with cocontrolled support. Then we define $CY^q_b(U,A):=C^q(U,A)/CX^q_b(U,A)$.

\begin{defn}
 If $X$ is a metric space, $A$ an abelian group and $q\ge 0$ then \emph{coarse cohomology} $HY^q_b(X,A)$ is defined to be the $q$th homology of the coarse cochain complex $(CY^q_b(X,A),d_q)_{q\ge 0}$.
\end{defn}

Subsets $U_1,\ldots,U_n$ of a subset $U\s X$ of a metric space form a \emph{coarse disjoint union of $U$} if they coarsely cover $U$ and every two elements are disjoint.

\begin{lem}
\label{lem:h0}
 If $U\s X$ is a subset of a metric space and $A$ an abelian group then
 \[
  HY^0_b(U,A)=\ker d_0=A(U)
 \]
\end{lem}
\begin{proof}
We compute $HY_b^0(U,A)$. Let $\varphi\in \ker d_0$ be a cocycle. Then $x\sim y$ if $d_0\varphi(x,y)=0$ defines an equivalence relation on $X$ with equivalence classes $(\varphi^{-1}(k))_{k\in A}$. The $\varphi^{-1}(k)$ form a coarse disjoint union since $d_0\varphi$ has cocontrolled support. We can assume all $\varphi^{-1}(k)$ are not bounded otherwise we substract a cochain with bounded (cocontrolled) support. Thus $\varphi$ is an element of $A(U)$. 

If we are given a coarse disjoint union $U_1,\ldots,U_n$ of $U$ and $(a_i)_{U_i}\in A(U)$ then we can assume the $U_i$ are disjoint and not bounded. Then
\begin{align*}
 \varphi:U&\to A\\
 x&\mapsto \{ a_i \qquad x\in U_i
\end{align*}
defines a cocycle in $CY^0_b(U,A)$.
\end{proof}

If $\alpha:X\to Y$ is a coarse map between metric spaces and $\varphi\in C^q(Y,A)$ a cochain then $\alpha^*(\varphi):= \varphi\circ \alpha^{\times q+1}$ defines a cochain in $C^q(X,A)$, specifically $C^q_{A_1,\ldots,A_n}(Y,A)$ is mapped to $C^q_{\alpha^{-1}(A_1),\ldots,\alpha^{-1}(A_n)}(X,A)$. If $\varphi$ has cocontrolled support so does $\alpha^*\varphi$. Thus there is a well-defined cochain map $\alpha^*:CY^q_b(Y,A)\to CY^q_b(X,A)$.

In particular an inclusion $U\s V$ of subsets induces a restriction map $i^*:CY^q_b(V,A)\to CY^q_b(U,A)$. Thus $CY^q_b(\cdot, A)$ forms a presheaf on $X_{ct}$.

\begin{lem}
\label{lem:cqsheaf}
 If $X$ is a metric space the presheaf $CY^q_b(\cdot,A)$ is sheaf on $X_{ct}$. 
 \end{lem}
 \begin{proof}
  Let $U_1,\ldots,U_n$ be a coarse cover of a subset $U\s X$. We show the identity axiom. Let $\varphi\in CY_b^q(U,A)$ be a section with $\varphi_R|_{U_i}=0$ for every $i$. By Lemma~\ref{lem:cccocontrolled} the set $V\coloneqq(U_1^{q+1}\cup\cdots\cup U_n^{q+1})^c$ is cocontrolled. Then 
  \[
  \varphi=\varphi|_{U_1}+\cdots+\varphi|_{U_n}+\varphi|_V
  \]
  as a finite sum of functions with cocontrolled support has cocontrolled support.
  
  We show the gluing axiom. Suppose $\varphi_i\in CY^q_b(U_i,A)$ are functions with $\varphi_i|_{U_j}=\varphi_j|_{U_i}$ for every $i,j$. Define a function 
  \begin{align*}
  \varphi:U^{q+1}&\to A\\
  (x_0,\ldots,x_q)&\mapsto\begin{cases}
                           \varphi_1(x_0,\ldots,x_q) & (x_0,\ldots,x_q)\in U_1^{q+1}\\
                           \varphi_2(x_0,\ldots,x_q) & (x_0,\ldots,x_q)\in U_2^{q+1}\setminus U_1^{q+1}\\
                           \vdots &\vdots\\
                           \varphi_n(x_0,\ldots,x_q) & (x_0,\ldots,x_q)\in U_n^{q+1}\cap (U_1^{q+1}\cup \cdots \cup U_{n-1}^{q+1})^c\\
                           0&\mbox{otherwise}.
                          \end{cases}
  \end{align*}
  If $\varphi_i\in C_{A_{i1},\ldots,A_{in_1}}(U_i,A)$ then  $\varphi\in C_{A_{11},\ldots,A_{1n_1},A_{21}\setminus U_1,\ldots, A_{n,n_n}\setminus(U_1\cup\cdots U_{n-1}),(U_1\cup \cdots \cup U_n)^c}(U,A)$. As can easily be seen the cochain $\varphi$ restricts to $\varphi_i$ for every $i$.
 \end{proof}
\begin{lem}
\label{lem:cccocontrolled}
 If $R\ge 0$ and $U_1,\ldots,U_n$ are a coarse cover of a subset $U\s X$ of a metric space then
 \[
  (U_1^{q+1}\cup\cdots\cup U_n^{q+1})^c\cap \Delta_R^q
 \]
is bounded.
\end{lem}
\begin{proof}
 If $f:\{1,\ldots,n\}\to\{0,\ldots,q\}$ is a function then denote
 \[
  A_f\coloneqq\bigcap_{i\in f^{-1}(0)}U_i^c\times\cdots\times\bigcap_{i\in f^{-1}(q)}U_i^c.
 \]
 Here the empty intersection denotes $U$.
Then
\begin{align*}
(U_1^{q+1}\cup\cdots\cup U_n^{q+1})^c\cap U^{q+1}
&=\{(x_0,\ldots,x_q)\in U^{q+1}:\forall i\in\{1,\ldots,n\}\exists j\in\{0,\ldots,q\}x_j\not\in U_i\}\\
&=\bigcup_{f:\{1,\ldots,n\}\to \{0,\ldots,q\}}A_f.
\end{align*}
Now the projection to the $i$th factor of $A_f\cap\Delta_R^q$ is
\begin{align*}
(A_f\cap\Delta_R^q)_i
&\s\{x\in U: d(x,\bigcap_{i\in f^{-1}(0)}U_i^c)\le R,\ldots,d(x,\bigcap_{i\in f^{-1}(q)}U_i^c)\le R\}\\
&=\Delta_R[\bigcap_{i\in f^{-1}(0)}U_i^c]\cap\cdots\cap \Delta_R[\bigcap_{i\in f^{-1}(q)}U_i^c]\\
&\s\bigcap_{i=1}^n\Delta_R[U_i^c]
\end{align*}
bounded. Since $(U_1^{q+1}\cup\cdots\cup U_n^{q+1})^c\cap U^{q+1}$ is a finite union of the $A_f$ this proves the claim.
\end{proof}

\begin{lem}
\label{lem:flabby}
 The sheaf $CY^q_b(\cdot,A)$ on $X_{ct}$ is flabby.
\end{lem}
\begin{proof}
 If $U\s V$ is an inclusion of subsets and $\varphi\in C^q(U,A)$ then there is a disjoint union $A_1\sqcup\cdots\sqcup A_n=U$ with $\varphi\in C_{A_1,\ldots,A_n}(U,A)$. Define 
 \begin{align*}
  \tilde \varphi:V^{q+1}&\to A\\
  (x_0,\ldots,x_q)&\mapsto\begin{cases}
                           0 &\exists i\in\{0,\ldots,q\}:x_i\not\in U\\
                           \varphi(x_0,\ldots,x_q) &\mbox{otherwise}.
                          \end{cases}
 \end{align*}
Then $\tilde\varphi\in C_{A_1,\ldots,A_n,U^c}(V,A)$ restricts to $\varphi$ on $U$.
\end{proof}

\begin{lem}
 The homology of $C^q(X,A)$ is concentrated in degree zero.
\end{lem}
\begin{proof}
 We compute $H^q(C^*(X,A))=\begin{cases}
                           A & q=0\\
                           0 & \mbox{otherwise}
                          \end{cases}$. If $q=0$ then
\begin{align*}
 \ker d_0
 &=\{\varphi:X\to A:d_0\varphi=0\}\\
 &=\{\varphi:X\to A:\varphi(x_1)-\varphi(x_0)=0\forall x_1,x_0\in X\}\\
 &=\{\varphi:X\to A \mbox{ constant}\}\\
 &=A
\end{align*}
If $q\ge 1$ and $\varphi\in\ker d_q$ then define the cochain
\begin{align*}
 \tilde\varphi:X^q&\to Z\\
 (x_0,\ldots,x_{q-1})&\mapsto \varphi(0,x_0,\ldots,x_{q-1}).
 \end{align*}
Here $0\in X$ is a fixed point. If $\varphi\in C^q_{A_1,\ldots,A_n}(X,A)$ then $\tilde\varphi\in C^{q-1}_{A_1,\ldots,A_n}(X,A)$. Then
\begin{align*}
d_{q-1}\tilde \varphi(x_0,\ldots,x_q)
&=\sum_{i=0}^q(-1)^i\varphi(0,x_0,\ldots,\hat x_i,\ldots,x_q)\\
&=\sum_{i=1}^{q+1}(-1)^{i+1}\varphi(0,x_0,\ldots,\hat x_{i-1},\ldots,x_q)+d_q\varphi(0,x_0,\ldots,x_q)\\
&=\varphi(x_0,\ldots,x_q)
\end{align*}
Thus higher homology of $C^q(X,Z)$ vanishes.
\end{proof}

We note an exact sequence of cochain complexes:
\[
 \xymatrix{
  \cdots\ar[r]
  &CX_b^{q-1}(U;A)\ar[r]\ar[d]
  & CX_b^q(U;A)\ar[r]\ar[d]
  & CX_b^{q+1}(U;A)\ar[r]\ar[d]
  &\cdots\\
  \cdots\ar[r]
  &C^{q-1}(U;A)\ar[r]\ar[d]
  & C^q(U;A)\ar[r]\ar[d]
  & C^{q+1}(U;A)\ar[r]\ar[d]
  &\cdots\\
  \cdots\ar[r]
  &CY_b^{q-1}(U;A)\ar[r]
  & CY_b^q(U;A)\ar[r]
  & CY_b^{q+1}(U;A)\ar[r]
  &\cdots
 }
\]

\begin{lem}
\label{lem:cycxexact}
 If $q\ge 1$ and for every subset $U\s X$ and $\psi\in \ker dX_{q+1}^b$ there exists a coarse cover $U_i$ of $U$ and $\psi_i\in CX_b^q(U_i;A)$ with $d_q(\psi_i)=\psi|_{U_i}$ then $CY^*_b(\cdot;A)$ is exact at $q$.  
\end{lem}
\begin{proof}
 Let $\varphi\in \ker dY^b_q\s CY_b^q(U;A)$ be an element. Then $\psi:=d_q\varphi$ has cocontrolled support, thus is an element of $CX^{q+1}_b(U;A)$. Since $d_{q+1}\psi=d_{q+1}d_q\varphi=0$ the element $\psi$ is even a cocycle in $CX^{q+1}_b(X;A)$. Then there exist a coarse cover $(U_i)_i$ and elements $\psi_i\in CX_b^q(U_i;A)$ with $d_q(\psi_i)=(d_q\varphi)|_{U_i}$. Then $\varphi|_{U_i}-\psi_i$ is a cocycle in $C^q(U_i;A)$ for every $i$. Thus there exists $\varphi_i\in C^{q-1}(U_i;A)$ with $d_{q-1}\varphi_i=\varphi|_{U_i}-\psi_i$. Thus $\varphi|_{U_i}$ represents a coboundary.
\end{proof}

\begin{lem}
\label{lem:cocontrolledcc}
 If $q\ge 1$ and $(U_1^{q+1}\cup\cdots\cup U_n^{q+1})^c$ is cocontrolled in $U^{q+1}$ then $U_1,\ldots,U_n$ is a coarse cover of $U$.
\end{lem}
\begin{proof}
 Let $R\ge 0$ be a number. If $(x,y)\in (\bigcup_iU^2)^c\cap \Delta_R$ then $(x,y,\ldots,y)\in (\bigcup_iU_i^{q+1})^c\cap \Delta_R^q$. Since this set is bounded $(x,y)$ must be contained in a bounded set too.
\end{proof}

\begin{lem}
\label{lem:exactoncx}
 If $\psi\in CX_b^q(U,A)$ is a cocycle then there exists a coarse cover $U_1,\ldots,U_n$ of $U$ with $\psi|_{U_i}=0$ for every $i=1,\ldots,n$.
\end{lem}
\begin{proof}
 Suppose $\psi\in C_{A_1,\ldots,A_n}(U,A)$ and fix $a_i\in A_i$ for every $i$. Then $A_i$ is bounded if $\psi(a_i,\ldots,a_i)\not=0$. We add $i$ to a list $\mathcal C$. Likewise $A_i\times A_j$ is cocontrolled if there exists a map $f:\{0,\ldots,q\}\to\{i,j\}$ with $\psi(a_{f(0)},\ldots,a_{f(q)})\not=0$. We add the set $\{i,j\}$ to the list $\mathcal C$. We proceed likewise with $A_i\times\cdots\times A_j$ of up to $q+1$ factors. We then define
 \[
  \mathcal U=\{A_{i_1}\cup\cdots\cup A_{i_m}\mid S\not\in \mathcal C\forall S\s \{i_1,\ldots,i_m\}\}
 \]
We show $\mathcal U$ is the desired coarse cover.

If $V\in \mathcal U$ then it is of the form $V=A_{i_1}\cup \cdots\cup A_{i_m}$. Let $f:\{0,\ldots,q\}\to\{i_1,\ldots,i_m\}$ be a function. Since $\{f(0),\ldots,f(q)\}\s \{i_1,\ldots,i_m\}$ we have $\{f(0),\ldots, f(q)\}\not\in \mathcal C$. Thus $\psi(a_{f(0)},\ldots,a_{f(q)})=0$. Since $f$ was arbitrary this implies $\psi|_V=0$.

If $(a_{i_0},\ldots,a_{i_q})\in (\bigcup_{V\in\mathcal U}V^{q+1})^c$ then $A_{i_0}\cup\cdots\cup A_{i_q}\not\in \mathcal U$. Thus there exists a subset $S\s\{i_0,\ldots,i_q\}$ with $S\in \mathcal C$. Which implies that $A_{i_0}\times\cdots\times A_{i_q}$ is cocontrolled. This way we showed that $(\bigcup_{V\in\mathcal U}V^{q+1})^c$ is cocontrolled in $U^{q+1}$. By Lemma~\ref{lem:cocontrolledcc} we can conclude that $\mathcal U$ is a coarse cover.
\end{proof}

\begin{thm}
 If $X$ is a metric space and $A$ an abelian group then the $CY^q_b(X,A)$ are a flabby resolution of the constant sheaf $A$ on $X_{ct}$. We can compute
 \[
  \cohomology q X A= HY_b^q(X,A).
 \]
for every $q\ge 0$.
\end{thm}
\begin{proof}
 We prove that
 \[
  0\to A\to CY_b^0(X,A)\to CY_b^1(X,A)\to\cdots
 \]
 is a flabby resolution of $A$. By Lemma~\ref{lem:cqsheaf} the $CY_b^q(X,A)$ are sheaves. They are flabby by Lemma~\ref{lem:flabby}. By Lemma~\ref{lem:h0} the sequence is exact at $0$. If we combine Lemma~\ref{lem:cycxexact} and Lemma~\ref{lem:exactoncx} then we see that it is exact for $q\ge 1$.
\end{proof}

\section{Functoriality, graded ring structure and Mayer-Vietoris}
This chapter presents a few immediate applications of Theorem~\ref{thm:introCochains}.

\begin{lem}
 If two coarse maps $\alpha,\beta:X\to Y$ are close then they induce the same map in cohomology.
\end{lem}
\begin{proof}
 The chain homotopy for coarse cohomology presented in \cite[Proposition~5.12]{Roe2003} can be costumized for our setting. Define a map $h:CY_b^q(Y,A)\to CY_b^{q-1}(X,A)$ by
 \[
  (h\varphi)(x_0,\ldots,x_{q-1})=\sum_{i=0}^{q-1}(-1)^i\varphi(\alpha(x_0),\ldots,\alpha(x_i),\beta(x_i),\ldots,\beta(x_{q-1})).
 \]
If $\varphi\in C^q_{A_1,\ldots,A_n}(Y,A)$ then $h\varphi\in C^{q-1}_{\alpha^{-1}(A_i)\cap\beta^{-1}(A_j)}(X,A)$. Since $\alpha,\beta$ are close, cocontrolled support of $\varphi$ implies cocontrolled support of $h\varphi$. Thus $h$ is well-defined.

The combinatorical calculation presented in the proof of \cite[Proposition~5.12]{Roe2003} shows that
\[
 d_{q-1}(h \varphi)+h(d_q\varphi)=\beta^*\varphi-\alpha^*\varphi.
\]
Thus $\alpha^*,\beta^*$ are cochain homotopic.
\end{proof}

Throughout this section $X$ denotes a metric space and $R$ a commutative ring.

We define a map on cochains
\begin{align*}
CY_b^p(X,R)\times CY_b^q(X,R)&\to CY_b^{p+q}(X,R)\\
(\phi,\psi)&\mapsto \phi\vee\psi
\end{align*}
with 
\[
(\phi\vee\psi)(x_0,\ldots,x_{p+q})=\phi(x_0,\ldots,x_p)\psi(x_p,\ldots,x_{p+q}).
\]

\begin{lem}
This product $\vee$ is well-defined. 
\end{lem}
\begin{proof}
We show $\phi\vee \psi$ has cocontrolled support if one of $\phi,\psi$ does.

Suppose $\phi$ has cocontrolled support. Then $\supp \phi\cap \Delta_R^p$ is bounded for every $R\ge 0$. This implies the $0$th factor $B:=(\supp \phi\cap\Delta_R^p)_0$ is bounded. If $p\le i\le p+q$ then the $i$th factor
\[
(\supp(\phi\vee\psi)\cap \Delta_R^{p+q})_i\s \{x\in X:d(x,B)\le R\}
\]
is bounded. This proves the claim.
\end{proof}

The formula 
\[
d_{p+q}(\phi\vee\psi)=d_p\phi\vee\psi + (-1)^p\phi\vee d_q\psi
\]
is easy to check. From that we deduce that the $\vee$-product of two cocycles is a cocycle and the product of a cocycle with a coboundary is a coboundary. Thus $\vee$ gives rise to a cup-product $\cup$ on $HY_b^*(X,R)$. Associativity and the distributive law can be checked on cochain level. This makes $(HY_b^*(X;R),+,\cup)$ a graded ring.

\begin{thm}
\label{thm:mv}
(Mayer-Vietoris) If $U_1,U_2$ is a coarse cover of a metric space $X$ then there is a long exact sequence in cohomology
 \begin{align*}
  0&\to\check H^0_{ct}(X;A)\to\check H^0_{ct}(U_1;A)\oplus \check H^0_{ct}(U_2;A)\to \check H^0_{ct}(U_1\cap U_2;A)\\
  &\to \check H^1_{ct}(X;A)\to\cdots\\
  &\to\check H^q_{ct}(X;A)\to\check H^q_{ct}(U_1;A)\oplus \check H^q_{ct}(U_2;A)\to \check H^0_{ct}(U_1\cap U_2;A)\\
  &\to \check H^{q+1}_{ct}(X;A)\to\cdots
 \end{align*}

\end{thm}
\begin{proof}
 We examine a sequence of cocomplexes
 \begin{align*}
  0\to CY_b^q(X;A)\xrightarrow{\alpha} CY_b^q(U_1;A)\oplus CY_b^q(U_2;A)\xrightarrow{\beta} CY_b^q(U_1\cap U_2;A)\to 0.
 \end{align*}
 Here $\alpha$ is defined by $\varphi\mapsto (\varphi|_{U_1},\varphi|_{U_2})$ and $\beta$ by $(\varphi_1,\varphi_2)\mapsto \varphi_2|_{U_1}-\varphi_1|_{U_2}$.
 This sequence is exact at $CY_b^q(X;A)$ and $CY_b^q(U_1;A)\oplus CY_b^q(U_2;A)$ since $CY_b^q(\cdot;A)$ is a sheaf on $X_{ct}$ and $U_1,U_2$ are a coarse cover of $X$. It is exact at $CY_b^q(U_1\cap U_2)$ since $CY_b^q(\cdot;A)$ is a flabby sheaf.
 
 The result is obtained by taking the long exact sequence in cohomology of the exact sequence of cochain complexes.
\end{proof}

We denote by $HX_b^q(X;A)$ the $q$th homology group of the cochain complex $CX_b^q(X;A)$ given a metric space $X$ and an abelian group $A$.

\begin{prop}
Let $X$ be a metric space and $A$ an abelian group. Then 
\[
 HX_b^0(X,A)=\begin{cases}
              A & e(X)=0\\
              0 & e(X)>0;
             \end{cases}
\]

 \[
  HX_b^1(X,A)=\begin{cases}
                        0 & e(X)=0\\
                        A^{e(X)-1} & 0<e(X)<\infty\\
                        \bigoplus_\N A & e(X)=\infty
            \end{cases}
 \]
and
\[
 HX_b^{q+1}(X,A)=HY_b^q(X,A)
\]
for every $q\ge 1$.
\end{prop}
\begin{proof}
Suppose $q=0$. An element of $\ker d_0$ is a constant function. If $X$ is bounded every constant function represents an element of $HX_b^0(X,A)=A$. If $X$ is not bounded then every constant function with bounded (cocontrolled) support must be zero.

 If $q>0$ we use the short exact sequence of cochain complexes
 \[
  0\to CX_b^q(X,A)\to C^q(X,A)\to CY_b^q(X,A)\to 0.
 \]
which splits since every element of $CY_b^q(X,A)$ is represented by a blocky map $\varphi:X^{q+1}\to A$.

Now we produce the long exact sequence in cohomology. The first few terms are
\[
 0\to HX_b^0(X,A)\to H^0(X,A)\to HY_b^0(X,A)\to HX^1(X,A)\to 0.
\]
If $X$ is bounded then this reads
\[
 0\to A\xrightarrow{id}A\xrightarrow{0}0\to HX^1(X,A)\to 0.
\]
Thus $HX_b^1(X,A)=0$. If $X$ is not bounded then the first few terms of the long exact sequence in cohomology read 
\[
 0\to 0\to A\to A(X)\to HX_b^1(X,A)\to 0
\]
Then 
\[
HX_b^1(X,A)=\mathcal C_f(X,A)/A=\begin{cases}
                                     A^{e(X)-1} & e(X)<\infty\\
                                     \bigoplus_\N A & e(X)=\infty
                                    \end{cases}
\]
In the middle term we mod out the constant functions.

For $q\ge 1$ the long exact in cohomology reads
\[
 0\to HY_b^q(X,Z)\to HX_b^{q+1}(X,Z)\to 0.
\]
Thus we proved the claimed results.
\end{proof}

\section{Computations}
\label{sec:computations}
Theorem~\ref{thm:introCochains} can also be applied to compute cohomology groups in a combinatorical manner.
\begin{lem}
\label{lem:h1z}
 If $A$ is a finite abelian group then $\check H_{ct}^1(\Z_{\ge 0},A)=0$.
\end{lem}
\begin{proof}
  If $\varphi\in \ker dY^b_1$ then $d_1\varphi$ has cocontrolled support. This means for every $n\in \N$ the set $\supp(d_1\varphi)\cap \Delta_n^1$ is bounded. This amounts to saying that the function
 \[
  \varphi_n:x\mapsto\varphi(x,x+n)-\varphi(x,x+1)-\cdots-\varphi(x+n-1,x+n)
 \]
 has finite support $x_1,\ldots,x_{r_n}\in \Z_{\ge 0}$.
Then define
\begin{align*}
 \tilde \varphi:\Z_{\ge0}&\to A\\
 x&\mapsto \sum_{i=0}^{x-1}\varphi(i,i+1).
\end{align*}
This function is blocky since $\tilde\varphi$ assumes only finitely many values. Here we use that $A$ is finite. Moreover we have
\begin{align*}
 d_0\tilde\varphi(x,x+n)
 &=\sum_{i=0}^{x+n-1}\varphi(i,i+1)-\sum_{i=0}^{x-1}\varphi(i,i+1)\\
 &=\sum_{i=x}^{x+n-1}\varphi(i,i+1)\\
 &=\begin{cases}
    \varphi(x,x+n)-\varphi_n(x) & x=x_1,\ldots,x_{r_n}\\
    \varphi(x,x+n)&\mbox{otherwise}.
   \end{cases}
\end{align*}
This shows that $\varphi-d_0\tilde \varphi$ has cocontrolled support. Thus $\varphi$ is a coboundary in $CY_b^1(\Z_{\ge 0},A)$.
\end{proof}

\begin{thm}
\label{thm:h1tree}
 If $T$ is a tree and $A$ is a finite abelian group then
 \[
  \check H_{ct}^1(T;A)=0.
 \]
\end{thm}
\begin{proof}
Designate an element $t_0\in T$ as the root of the tree and define
\[
 S=\{t\in T\mid d(t_0,t)\in\N_0\}.
\]
Let $\varphi\in CY^1_b(S,A)$ be a cocycle. Then for every $s\in S$ there is a unique $1$-path $a_0=t_0,\ldots,a_n=s$ joining $t_0$ to $s$. Define
\begin{align*}
 \tilde\varphi:S&\to A\\
 s&\mapsto \sum_{i=0}^{n-1}\varphi(a_i,a_{i+1})
\end{align*}
Let $n\ge 0$ be a number. If $s,t\in S$ and $d(s,t)=n$ then there exist paths $a_0,\ldots,a_k,c_0,\ldots,c_m$ joining $t_0$ to $t$ and $a_0,\ldots,a_k, b_0,\ldots,b_l$ joining $t_0$ to $s$ with $n=m+l+2$ or $n=m+l$ depending on whether $c_m,c_{m-1},\ldots,c_0,a_k,b_0,b_1,\ldots,b_l$ or $c_m,c_{m-1},\ldots,c_0,b_0,\ldots,b_l$ is a $1$-path joining $t$ to $s$.
Then
\[
 d_0\tilde\varphi(s,t)-\varphi(s,t)=\varphi(a_k,c_0)+\sum_{i=0}^{m-1}\varphi(c_i,c_{i+1})-\varphi(a_k,b_0)-\sum_{i=0}^{l-1}\varphi(b_i,b_{i+1})-\varphi(s,t)
\]
is a would-be cocycle in $\Delta^1_{n+2}$ thus has bounded support. Thus we showed $\varphi$ is a coboundary.

Since the inclusion $S\s T$ is coarsely surjective we can conclude $\check H_{ct}^1(T,A)=\check H_{ct}^1(S,A)=0$.

\end{proof}

\begin{lem}
 $\check H_{ct}^1(\Z^2,\Z/3\Z)\ge \Z/3\Z$.
\end{lem}
\begin{proof}
 We divide $\Z^2$ into 4 quadrants: $A_0=\Z_{\ge0}\times\Z_{\ge0},A_1=\Z_{\ge0}\times \Z_{<0},A_2=\Z_{<0}\times\Z_{<0},A_3=\Z_{<0}\times\Z_{\ge0}$. Fix points $0\in A_0,\ldots,3\in A_3$. We define a cocycle $\varphi\in CY^1_b(\Z^2,\Z/3\Z)$ by
 \[
 \varphi=\begin{array}{c|cccc}
   &0&1&2&3\\
   \hline
  0&0&1&2&1\\
  1&2&0&1&2\\
  2&2&2&0&1\\
  3&2&2&2&0
 \end{array}.
\]
Among other equations we obtain $d_1\varphi(0,2,0)=2-0+2=1$ and $d_1\varphi(1,3,1)=2-0+2=1$. Thus $A_0\times A_2\times A_0$ and $A_1\times A_3\times A_1$ lie in the support of $d_1\varphi$. Since $A_0\times A_2$ and $A_1\times A_3$ are cocontrolled this is okay. Checking the other finitely many equations we ob $d_1\varphi$ does indeed have cocontrolled support. Thus $\varphi$ is indeed a cocycle. We show $\varphi$ is not a coboundary. Suppose for contradiction that $\varphi=d_0\psi+\tilde \varphi$ with $\psi\in CY_b^0(\Z^2,\Z/3\Z)$ and $\tilde \varphi\in CX_b^1(\Z^2,\Z/2\Z)$. First we show $\psi\in C_{A_0,\ldots,A_3}(\Z^2,\Z/3\Z)$. Suppose for contradiction $A_0=A_{01}\sqcup A_{02}$ and $\psi|_{A_{01}}\equiv k$ while $\psi|_{A_{02}}$ does not take the value $k$. Then $A_{01}\times A_{02}\s \supp d_0\psi=\supp(\varphi-\tilde \varphi)$. Since $\varphi|_{A_0\times A_0}\equiv 0$ this implies $A_{01}\times A_{02}$ is contained in the support of $\tilde \varphi$ and therefore is cocontrolled. Since $A_0$ is oneended this is a contradiction. Now we construct $\psi$. Suppose $\psi(0)=k\in\Z/3\Z$. Since $A_0\times A_1$ is not cocontrolled we obtain $\psi(1)=\psi(1)-\psi(0)+\psi(0)=\varphi(0,1)+\psi(0)=2+k$ and similarly $\psi(3)=2+k$. Since $A_1\times A_2$ is not cocontrolled we obtain $\psi(2)=\psi(2)-\psi(1)+\psi(1)=\varphi(1,2)+\psi(1)=1+k$. Then
\[
 d_0\psi=\begin{array}{c|cccc}
           &0&1&2&3\\
           \hline
          0&0&1&2&1\\
          1&2&0&1&0\\
          2&1&2&0&2\\
          3&2&0&1&0
         \end{array}
\]
If we compare the tables we see that $d_0\psi-\varphi$ does not have cocontrolled support. Thus $\varphi$ is a proper cocycle.
\end{proof}

\section{Infinite coefficients}
The computations in Chapter~\ref{sec:computations} only work for finite coefficients. This chapter shows that the coefficient $\Z$ does not produce interesting cohomology groups.

If $X$ is a proper metric space and $A$ a metric space denote by $\mathcal C_h(X,A)$ the abelian group of Higson functions $\varphi:X\to A$ modulo functions with bounded support. Namely a bounded function $\varphi:X\to A$ is called Higson if for every entourage $E\s X\times X$ and every $\varepsilon>0$ there exists a compact subset $K\s X$ such that $(x,y)\in E\setminus(K\times K)$ implies $d(\varphi(x),\varphi(y))<\varepsilon$.

\begin{lem}
 If $X$ is a proper $R$-discrete for some $R>0$ metric space and $A$ a metric space then $U\mapsto \mathcal C_h(U,A)$ for $U\s X$ with the obvious restriction maps is a sheaf on $X_{ct}$.
\end{lem}
\begin{proof}
 Let $U_1,\ldots,U_n$ be a coarse cover of a subspace $U\s X$.
 
 We prove the base identity axiom. Let $\varphi\in \mathcal C_h(U,A)$ be an element. If $\varphi|_{U_1},\ldots,\varphi|_{U_n}$ have bounded support $B_1,\ldots,B_n$ then $\varphi$ has support contained in the set
 \[
  B_1\cup\cdots\cup B_n\cup (U_1\cup\cdots\cup U_n)^c
 \]
which is a finite union of bounded sets and therefore itself bounded. Thus $\varphi$ has bounded support.

Now we prove the gluing axiom. Suppose there are functions $\varphi_1\in\mathcal C_h(U_1,A),\ldots,\varphi_n\in \mathcal C_h(U_n,A)$ with $\varphi_i|_{U_j}=\varphi_j|_{U_i}$ for every $i,j\in \{1,\ldots,n\}$. Then define a function
\begin{align*}
\varphi:U&\to A\\
x&\mapsto\begin{cases}
          \varphi_1(x) & x\in U_1\\
          \varphi_2(x) & x\in U_2\setminus U_1\\
          \vdots &\vdots\\
          \varphi_n(x) & x\in U_n\cap U_1^c\cap\cdots \cap U_{n-1}^c\\
          0 & x\in U_1^c\cap\cdots\cap U_n^c
         \end{cases}
\end{align*}
which restricts to $\varphi_i$ on each $U_i$. Here $0\in A$ is any choice of point. Now $\varphi$ is continuous since $U$  is $R$-discrete. The image of $\varphi $ 
\[
 \im\varphi=\im \varphi_1\cup\cdots\cup\im \varphi_n\cup\{0\}
\]
is bounded. Now we check the Higson property. If $E\s U^2$ is an entourage then
\[
 E=E|_{U_1^2}\cup\cdots\cup E|_{U_n^2}\cup E|_{B^2}
\]
where $B\subseteq U$ is bounded. Let $\varepsilon>0$ be a number. Then for each $i$ there exists a bounded subset $K_i\subseteq U_i$ with $|\varphi_i(x)-\varphi_i(y)|<\varepsilon$ for each $(x,y)\in E|_{U_i^2}\setminus K_i^2$. Now define $K=K_1\cup\cdots\cup K_n\cup B$. Then $|\varphi(x)-\varphi(y)|<\varepsilon$ for each $(x,y)\in E\setminus K^2$. Thus $\varphi$ satisfies the Higson property. 
\end{proof}

Denote by $\mathcal C_f(X,A)$ the abelian group of Freudenthal functions $\varphi:X\to A$ modulo functions with bounded support. Namely a bounded function $\varphi:X\to A$ is called Freudenthal if for every entourage $E\s X\times X$ there exists a compact subset $K\s X$ such that $(x,y)\in E\setminus(K\times K)$ implies $\varphi(x)=\varphi(y)$.

\begin{lem}
 If $X$ is a proper metric space and $A$ is a countable abelian group then the constant sheaf $A$ on $X_{ct}$ is isomorphic to $\mathcal C_h(\cdot,A)$ which is isomorphic to $\mathcal C_f(\cdot, A)$.
\end{lem}
\begin{proof}
 For each $U\s X$ we show the inclusion $\mathcal C_f(U,A)\to \mathcal C_h(U,A)$ is bijective. Let $\varphi\in \mathcal C_h(U,A)$ be an element. We show $\varphi$ is Freudenthal. Let $E\s U^2$ be an entourage and choose $\varepsilon=1/2$. Then there exists a bounded set $K_\varepsilon\s U$ with $d(\varphi(x),\varphi(y))<1/2$ for each $(x,y)\in E\setminus K^2$. Since $A$ is $1$-discrete this implies $\varphi(x)=\varphi(y)$ for each $(x,y)\in E\setminus K^2$. Thus $\varphi$ is Freudenthal.
 
 Now we show the constant sheaf $A$ is isomorphic to $\mathcal C_f(\cdot,A)$. We define a homomorphism $\Phi:A(U)\to \mathcal C_f(U,A)$. If $a\in A(U)$ then it can be written $a=a_1\oplus\cdots\oplus a_n$ corresponding to a coarse disjoint union $U_1\sqcup \cdots \sqcup U_n$. Then we define
 \begin{align*}
  \varphi_i:U_i&\to A\\
  x&\mapsto a_i
 \end{align*}
The $\varphi_i$ are Higson and glue to a Higson function $\Phi(a):U\to A$. If $b\in A(U)$ is another element represented by $b_1\oplus\cdots\oplus b_m$ corresponding to $V_1\sqcup\cdots\sqcup V_m$ then $a+b$ is represented by $\bigoplus_{ij} (a_i+b_j)$ corresponding to $(U_i\cap V_j)_{ij}$. Without loss of generality the $U_i$ are pairwise disjoint and cover $U$ and likewise the $V_j$ are pairwise disjoint and cover $U$. Then
\[
(\Phi(a)+\Phi(b))(x)=\begin{cases}
                      a_1 & x\in U_1\\
                      \vdots & \vdots\\
                      a_n & x\in U_n
                     \end{cases}+\begin{cases}
                      b_1 & x\in V_1\\
                      \vdots & \vdots\\
                      b_m & x\in V_m
                     \end{cases}=\begin{cases}
                      a_1+b_1 & x\in U_1\cap V_1\\
                      \vdots & \vdots\\
                      a_n+b_m & x\in U_n\cap V_m
                     \end{cases}=\Phi(a+b) 
\]
Thus $\Phi$ is a homomorphism. We show $\Phi$ is well defined. Suppose $c\in A(U)$ is represented by $0$ on $B^c$ where $B$ is finite. Then $\Phi(c)$ has bounded support. Thus $\Phi$ is welldefined.

Now we construct the inverse $\Psi:\mathcal C_f(U,A)\to A(U)$. If $\varphi\in \mathcal C_f(U,A)$ is Freudenthal then in particular its image $\im\varphi=\{a_1,\ldots,a_n\}$ is finite. Since $\varphi$ is Higson the $\varphi^{-1}(a_1),\ldots,\varphi^{-1}(a_n)$ are a pairwise coarsely disjoint union of $U$. Then define $\Psi(\varphi)$ to be the element represented by $a_1\oplus\cdots\oplus a_n$ corresponding to $\iip \varphi {a_1},\ldots,\iip \varphi{a_n}$. It is mapped by $\Phi$ to $\varphi$. Thus $\Phi$ is surjective. Now we show $\Phi$ is injective: If $\varphi$ has bounded support $B$ then $\Psi(\varphi)$ is represented by $0$ on $\iip \varphi {B^c}$.
\end{proof}

\begin{prop}
 If $X$ is a proper metric space the following sequence of sheaves on $X_{ct}$ is exact:
 \[
  0\to \mathcal C_f(\cdot,\Z)\to \mathcal C_h(\cdot,\R)\to \mathcal C_h(\cdot,S^1)\to 0
 \]
\end{prop}
\begin{proof}
 Let $U\s X$ be a subset. 
 
 The map $\mathcal C_f(U,\Z)\to \mathcal C_h(U,\R)$ is induced by the inclusion $\Z\to \R$. This map is well-defined since every Freudenthal function is Higson. It is injective. Thus exactness at $\mathcal C_f(\cdot,\Z)$ is guaranteed.
 
 The map $\mathcal C_h(U,\R)\to \mathcal C_h(U,S^1)$ is induced by the quotient map $\R\to\R/\Z$. A Higson function is in the kernel of this map exactly when its image is contained in $\Z$. Thus exactness at $\mathcal C_h(\cdot,\R)$ is ensured.
 
 Now we show exactness at $\mathcal C_h(\cdot,S^1)$. Let $\varphi\in \mathcal C_h(\cdot,S^1)$ be a function. Its image $\R/\Z$ is covered by $V_1\coloneqq[0.25,1]+\Z,V_2\coloneqq[0.75,1.5]+\Z$. Since 
 \[
  \overline{V_1^c}\cap\overline{V_2^c}=([0,0.25]+\Z)\cap([0.5,0.75]+\Z)=\emptyset
 \]
we obtain that $\varphi^{-1}(V_1^c)\notclose \varphi^{-1}(V_2^c)$ are coarsely disjoint. Thus $U_1\coloneqq\varphi^{-1}(V_1),U_2\coloneqq\varphi^{-1}(V_2)$ are a coarse cover of $U$. Now we describe a lift $\varphi_1:U_1\to \R$ of $\varphi|_{U_1}$: If $x\in U_1$ then $\varphi_1(x)$ is defined to be the representative of $\varphi(x)$ in the interval $[0.25,1]$. A lift $\varphi_2:U_2\to \R$ of $\varphi|_{U_2}$ is obtained by defining $\varphi_2(x)$ to be the representative of $\varphi(x)$ in the interval $[0.75,1.5]$. Thus the right morphism in the diagram is surjective. 
\end{proof}

\begin{lem}
 If $X$ is a proper metric space the sheaf $\mathcal C_h(\cdot,\R)$ is flabby. 
\end{lem}
\begin{proof}
Let $A\s X$ be a subset. Since $A\s hX$ is a subset the closure of $A$ in $hX$ is a compactification $\bar A$ generated by $C_h(X)|_A\coloneqq\{\varphi|_A:\varphi\in C_h(X)\}$. Now $\bar A$ is equivalent (as a compactification) to $hA$ and thus also generated by $C_h(A)$. Since both $C_h(X)|_A$ and $C_h(A)$ separate points from closed sets they are both contained in $C_h$, the algebra of bounded functions on $A$ which extend to $hA$ \cite[Proposition~2]{Mendivil1995}. The \cite[Proposition~2]{Mendivil1995} also states that $C_h\s C^*(A)$ is the smallest unital, closed $C^*$-algebra with this property. Since both $C_h(X)|_A$ and $C_h(A)$ are unital, closed the equality
\[
 C_h(X)|_A=C_h=C_h(A)
\]
holds.

We provide another proof using the Tietze extension theorem. If an element in $\mathcal C_h(U,\R)$ is represented by $\varphi\in \mathcal C_h(U)$ then it extends to $\bar\varphi$ on $hU$. By the Tietze extension theorem we can extend $\bar\varphi$ to a bounded function $\hat \varphi$ on $hX$. Then $\hat \varphi|_X$ represents an element in $\mathcal C_h(X,\R)$ that restricts to $\varphi$.
\end{proof}

\begin{ex}
We show $\mathcal C_h(\cdot,\R)$ is flabby on the specific example $\Z$ constructing a concrete global lift of a Higson function $\varphi:U\to \R$  on a subspace of $\Z$. If $z\in U^c$ then there are $z_-,z_+\in U$ with $z_-$ the largest number in $U$ with $z_-<z$ and $z_+$ the smallest number in $U$ with $z<z_+$. Define
\[
 \bar \varphi(z)=\frac{\varphi(z_-)(z_+-z)+\varphi(z_+)(z-z_-)}{z_+-z_-}.
\]
This function $\bar \varphi$ is Higson: If $\varepsilon>0$ then there exists an $N\in \N$ with $|\varphi(y)-\varphi(y')|<\varepsilon$ for every $y,y'\in U$ with $|y-y'|<2|\varphi|/\varepsilon$ and $y,y'>N$. If $z,z'\in U^c$ with $|z-z'|=1$ then 
\begin{align*}
 |\bar\varphi(z)-\bar\varphi(z')|
 &=\left|\frac{\varphi(z_-)(z_+-z)+\varphi(z_+)(z-z_-)}{z_+-z_-}-\frac{\varphi(z_-)(z_+-z')+\varphi(z_+)(z'-z_-)}{z_+-z_-}\right|\\
 &=\left|\frac{\varphi(z_-)(z'-z)+\varphi(z_+)(z-z')}{z_+-z_-}\right|\\
 &=\left|\frac{\varphi(z_+)-\varphi(z_-)}{z_+-z_-}\right|\\
 &<\begin{cases}
    \frac{2|\varphi|}{2|\varphi|/\varepsilon} & z_+-z_->2|\varphi|/\varepsilon\\
    \frac{\varepsilon}{z_+-z_-} & z_+-z_-\le 2|\varphi|/\varepsilon
   \end{cases}\\
&\le \varepsilon
\end{align*}
provided $z,z'>N$.

\end{ex}

\begin{rem}
By the long exact sequence in cohomology we obtain
\[
\check H_{ct}^1(X;A)=\frac{\mathcal C_h(X,S^1)}{\mathcal C_h(X,\R)}.
\]
If $X=\Z$ we can define
\begin{align*}
\varphi:\Z&\to\R\\
z&\mapsto\sum_{n=1}^{|z|}\frac 1n.
\end{align*}
This function satisfies the Higson condition but is not bounded. Post-composition with the projection $\R\to S^1$ gives a Higson function $\tilde \varphi:\Z\to S^1$ which does not have a lift. Compare this result with \cite{Keesling1994}.
\end{rem}

\begin{rem}
 It would be great if we could find an algorithm that computes coarse cohomology with constant coefficients of a finitely presented group. This does not work not even in degree $0$. If we could decide whether $\check H_{ct}^0(G;A)$ vanishes then we can decide whether $G$ is finite. This is in general not decidable.
\end{rem}

\section{The inverse image functor}

In this section we fix a coarse map $\alpha:X\to Y$ between metric spaces.

If $\sheafg$ is a sheaf on $Y_{ct}$ then the inverse image (or pullback sheaf) $\alpha^*\sheafg$ is the sheafification of the presheaf on $X_{ct}$ which assigns $U\s X$ with $\sheafg(\alpha(U))$.

Conversely if $\sheaff$ is a sheaf on $X_{ct}$ then the direct image $\alpha_*\sheaff$ is the sheaf on $Y_{ct}$ which assigns $\sheaff(\alpha^{-1}(V))$ to $V\s Y$.

\begin{lem}
 The functor $\alpha^*$ is left adjoint to $\alpha_*$. The functor $\alpha_*$ is left exact and the functor $\alpha^*$ is exact. The functor $\alpha_*$ maps injectives to injectives.
\end{lem}
\begin{proof}
 The functor
 \begin{align*}
  \ii \alpha:Cat(Y_{ct})&\to Cat(X_{ct})\\
  V&\mapsto \iip \alpha V
 \end{align*}
is a morphism of Grothendieck topologies and therefore gives rise to functors $(\ii \alpha)^p,(\ii \alpha)_p$ between categories of presheaves \cite[Chapter~I,2.3]{Tamme1994}. The functor $(\ii \alpha)^p$ maps a presheaf $\mathcal F$ on $X_{ct}$ to the presheaf $V\mapsto \sheaff(\iip \alpha V)$ on $Y_{ct}$. $(\ii \alpha)_p$ is defined in \cite[Theorem~I,2.3.1]{Tamme1994}. If $\sheafg$ is a presheaf on $ Y_{ct}$ and $U\s X$ we define $(\ii \alpha)_p\sheafg(U)$: Consider all $V\in Cat(Y_{ct})$ with $U\s \iip \alpha V$. They form a category $\mathcal I_U$. Then
\[
 (\ii \alpha)_p\sheafg(U)=\varinjlim_{V\in \mathcal I^{op}_U}\sheafg(V)=\sheafg(\alpha(U))
\]
since $\alpha(U)$ is the initial object in $\mathcal I_U$.

Then \cite[Chapter~I,3.6]{Tamme1994} discusses functors $(\ii \alpha)^s$ and $(\ii \alpha)_s$ between categories of sheaves. We obtain the direct image functor $\alpha_*=(\ii \alpha)^s$ and the inverse image functor $\alpha^*=(\alpha^{-1})_s$. Then \cite[Proposition~I,3.6.2]{Tamme1994} implies that $\alpha^*$ is left adjoint to $\alpha_*$, the functor $\alpha_*$ is left exact and if $\alpha^*$ is actually exact then the functor $\alpha_*$ maps injectives to injectives. It remains to show that $\alpha^*$ is exact. By \cite[Proposition~I,3.6.7]{Tamme1994} the functor $\alpha^*$ is exact if $\alpha^{-1}$ preserves finite fibre products and final objects. Indeed the inverse image of an intersection is an intersection of inverse images and the inverse image of the whole space is the whole space.
\end{proof}

\begin{prop}
 If we equip $\N$ with the topological coarse structure associated to the one-point compactification $\N^+$ of $\N$ we obtain a coarse space called $\ast$. This space $\ast$ is not metrizable but a final object for metric spaces. The constant sheaf on $\ast$ is flabby. 
\end{prop}
\begin{proof}
 A set $C\s\N^+$ is closed if it is finite or contains $+$. Thus a subset $E\s \N\times \N$ is an entourage if for every subset $E'\s E$ the projection of $E'$ to the first factor is finite exactly when the projection of $E'$ to the second factor is finite.
 
 Let $X$ be a metric space and $x_0\in X$ a basepoint. We define a map 
 \begin{align*}
  \rho:X&\to \ast\\
  x&\mapsto\lfloor d(x,x_0)\rfloor.
 \end{align*}
 We show this map is coarsely uniform: If $R\ge0$ and $F\s \rho^{\times 2}(\Delta_R)$ a subset such that the projection to the first factor is finite then choose arbitrary $(x,y)\in \Delta_R$ with $(\rho(x),\rho(y))\in F$. Since the first factor of $F$ is finite there is some $S\ge 0$ with $\lfloor d(x,x_0)\rfloor \le S$. Then 
 \begin{align*}
  \lfloor d(y,x_0)\rfloor
  &\le d(y,x)+d(x,x_0)+1\\
  &\le R+S+2.
 \end{align*}
Thus the projection of $F$ to the second factor is finite. This implies $\rho$ is coarsely uniform. If $B\s \ast $ is bounded then there exists some $S\ge 0$ such that $b\in B$ implies $b\le S-1$. Then $\rho^{-1}(B)$ is contained in a ball of diameter $S$ around $x_0$. Thus $\rho$ is coarsely proper. This way we showed $\rho $ is a coarse map. Suppose $\varphi:X\to \ast $ is another coarse map. Let $H'\s (\rho\times \varphi)(\Delta_0)$ be a subset such that the projection of $H'$ to the first factor is finite. We have that $H'$ is of the form 
\[
 H'=\{(\rho(x),\varphi(x))\mid x\in A\}
\]
for some subset $A\s X$. Then the projection of $H'$ to the first factor is $\rho(A)$. Since $\rho$ is coarsely proper the set $A\s \rho^{-1}\circ \rho(A)$ is bounded. This implies $\varphi(A)$ is bounded which is the projection of $H'$ to the second factor. Since we only used that $\rho,\varphi$ are coarse maps we can use the same argument with the factors reversed. Thus we showed that $(\rho\times\varphi)(\Delta_0)$ is an entourage in $\ast$. This implies $\rho,\varphi$ are close, they represent the same coarse map. This way we showed that $\ast$ is a final object for metric spaces.

If $A,B\s \ast$ are infinite subspaces then there exists a bijection $\varphi:A\to B$. The set $E:=\{(a,\varphi(a))\mid a\in A\}$ is an entourage and  $(A\times B)\cap E=E$ is not bounded. Thus $A\times B$ is not cocontrolled. We conclude that a coarse cover of a subset $U\s\ast$ contains an element which is cofinite in $U$. Let $A$ be an abelian group. Then
\[
 A_\ast(U)=\begin{cases}
            A & U \mbox{ infinite}\\
            0 & U \mbox{ finite}
           \end{cases}
\]
This shows that $A_\ast$ is flabby.
\end{proof}

\begin{lem}
 If $A$ is an abelian group then $A_X=\alpha^*A_Y$. The unit of $\alpha^*,\alpha_*$ at $A_Y$ is given by
 \begin{align*}
  \eta_{A_Y}(V):A_Y(V)&\to A_X(\alpha^{-1}(V))\\
  (\varphi:V\to A)&\mapsto(\varphi\circ\alpha:\iip \alpha V\to A).
 \end{align*}

\end{lem}
\begin{proof}
If $Z$ is a metric space and $\rho_Z:Z\to \ast$ the unique coarse map we prove $A_Z=\rho_Z^*A_\ast$. This proves the claim since 
\[
\alpha^*A_Y=\alpha^*\circ\rho_Y^*A_\ast=(\rho_Y\circ\alpha)^*A_\ast=\rho_X^*A_\ast=A_X.
\]
The sheaf $\rho_Z^*A_\ast$ is the sheafification of the following presheaf
\[
 U\mapsto A_\ast(\rho(U))=\begin{cases}
                           A & U\mbox{ not bounded}\\
                           0 & U\mbox{ bounded}
                          \end{cases}.
\]
Now this is just the constant sheaf on $Z$.

Now we compute the unit of the adjunction $\alpha^*,\alpha_*$. We denote by $\alpha^{-1}$ the presheaf inverse image functor. Then the unit of the adjunction $\alpha^{-1},\alpha_*$ at $A_Y$ is given by
\begin{align*}
 \eta^1_{A_Y}(V):A_Y(V)&\to \alpha_*\alpha^{-1}A_Y(V)=A_Y(\alpha\circ\alpha^{-1}(V))\\
 \varphi&\mapsto \varphi|_{\alpha\circ\alpha^{-1}(W)}
\end{align*}
The unit of the adjunction sheafification $\#$ and inclusion $\iota$ of presheaves in sheaves at $\alpha^{-1}A_Y$ is given by
\begin{align*}
 \eta^2_{\alpha^{-1}A_Y}(U):\alpha^{-1}A_Y(U)=A_Y(\alpha(U))&\to \alpha^*A_Y(U)=A_X(U)\\
 (\varphi:\alpha(U)\to A)&\mapsto(\varphi\circ\alpha:U\to A)
\end{align*}
This makes sense since $\varphi$ assigns a value $a_{V_i}$ to $V_i\s \alpha(U)$ where $(V_i)_i$ is a coarse disjoint union of $\alpha(U)$. Since $\alpha$ is a coarse map the $\alpha^{-1}(V_i)$ form a coarse disjoint union of $U$. Then $(a_{V_i})_{\alpha^{-1}(V_i)}$ represents $\varphi$ on $U$ which in cocycle notation is $\varphi\circ \alpha|_U$. 
Now we compose the units:
\[
 A_Y\xrightarrow{\eta^1_{A_Y}}\alpha_*\alpha^{-1}A_Y\xrightarrow{\alpha_*(\eta^2_{\alpha^{-1}A_Y})}\alpha_*\alpha^*A_Y
\]
and obtain the desired result.
\end{proof}

\begin{thm}
 The map induced by the inverse image functor 
 \[
  \alpha^*:\check H_{ct}^q(Y,A_Y)\to \check H_{ct}^q(X,A_X)
 \]
coincides with the canonical map 
\[
 \alpha^*:HY^q_b(Y,A)\to HY^q_b(X,A).
\]
\end{thm}
\begin{proof}
We apply \cite[Scolium~II.5.2]{Iversen1984}. We checked the proof of this result also works for sheaves on a Grothendieck topology. We choose $f=\alpha$, $G=A_Y$ and $T^q=CY^q(\cdot,A)$ a resolution on $Y$. Then $\alpha^*G=A_X$ has a resolution $CY^q(\cdot,A)$ on $X$. The morphism of complexes is given by 
 \begin{align*}
  \psi^q:CY^q_b(V,A)&\to CY^q_b(\iip \alpha V,A)\\
  (\varphi:V^{q+1}\to A)&\mapsto (\varphi\circ \alpha^{\times (q+1)}:(\iip \alpha U)^{q+1}\to A)
 \end{align*}
 which makes the diagram
\[
 \xymatrix{
 A_Y(V)\ar[r]\ar[d]_{\eta_{A_Y}(V)}
 & CY^*_b(V,A)\ar[d]^\psi\\
 A_X(\iip \alpha V)\ar[r]
 &CY^*_b(\iip \alpha V,A)
 }
\]
commute.
\end{proof}

\section{Sheaves on the Higson corona}
This chapter proves Theorem~\ref{thm:introHigson}.
\begin{lem}
\label{lem:refinement}
 Let $X$ be a metric space. If $U_1,\ldots, U_n$ is a coarse cover of $X$ then there exists a coarse cover $V_1,\ldots,V_n$ of $X$ with $V_i\notclose U_i^c$ for every $i=1,\ldots,n$.
\end{lem}
\begin{proof}
 By \cite[Lemma~15]{Hartmann2017b} there exists a cover $W_1,\cdots,W_n$ of $X$ as a set such that $W_i\notclose U_i^c$ for every $i$. For every $i$ there exists an in between set $C_i$ with $W_i\notclose C_i^c$ and $C_i\notclose U_i^c$. Then for every $i$ the sets $A_i^1:=W_i^c,A_i^2:=C_i$ are a coarse cover. Taking the intersection over those coarse covers provides a coarse cover
 \[
  \mathcal B:=\{A_1^{\varepsilon_1}\cap\cdots\cap A_n^{\varepsilon_n}:\epsilon_i=1,2\}.
 \]
Now let $B:=A_1^{\varepsilon_1}\cap\cdots\cap A_n^{\varepsilon_n}\in\mathcal B$ be an element. If there is some $i$ with $\varepsilon_i=2$ then
\[
 B\subseteq C_i\notclose U_i^c
\]
and in the other case $\varepsilon_i=1$ for every $i$. Thus
\[
 B=W_1^c\cap\cdots \cap W_n^c=\left(\bigcup_i W_i\right)^c=\emptyset.
\]
Now we join appropriate elements of $\mathcal B$ and obtain the desired coarse cover:
\[
 V_i:=\bigcup_{(\varepsilon_1,\ldots,\varepsilon_n)\in\{1,2\}^n,\varepsilon_i=2}A_1^{\varepsilon_1}\cap\cdots\cap A_n^{\varepsilon_n}.
\]
\end{proof}

Given a sheaf $F$ on $X_{ct}$ we define a sheaf $F^\nu$ on $\nu(X)$: If both $A\notclose U^c, B\notclose U^c$ then $A\cup B\notclose U^c$. Thus $(A)_{A\notclose U^c}$ is a directed poset by inclusion. Now we define a sheaf $F^\nu$ on basic open subsets of $\nu(X)$. If $U\subseteq X$ is a subset then
 \[
  F^\nu(\closedop{U^c}^c):=\varprojlim_{A\notclose U^c}F(A).
 \]
 If $V\subseteq U$ is an inclusion of subspaces then $A\notclose V^c$ implies $A\notclose U^c$. Thus there is a well-defined restriction map $F^\nu(\closedop {U^c}^c)\to F^\nu(\closedop{V^c}^c)$ which maps $(\varphi_A)_{A\notclose U^c}$ to $(\varphi_A)_{A\notclose V^c}$. This makes $F^\nu$ a presheaf.

\begin{prop}
\label{prop:catofsh1}
 If $X$ is a proper metric space and $F$  a sheaf on $X_{ct}$ then $F^{\nu}$ is a sheaf on $\nu(X)$.
\end{prop}
\begin{proof}
 Let $\closedop{U^c}^c=\bigcup_i\closedop{U_i^c}^c$ be a cover of a basic open set by basic open sets.
  
Let $A\subseteq X$ be a subset with $A\notclose U^c$. Then
 \[
  \closedop A\cap \bigcap_i \closedop{U_i^c}=\closedop A\cap\closedop{U^c}=\emptyset
 \]
Thus $(\closedop{U_i^c}^c)_i,\closedop A^c$ is an open cover of $\nu(X)$. Since $\nu(X)$ is compact there exists a finite subcover $\closedop{U_1^c}^c\cup\cdots\cup\closedop{U_n^c}^c\cup\closedop A^c$. By~\cite[Lemma~32]{Hartmann2019a} the subsets $U_1,\ldots,U_n,A^c$ are a coarse cover of $X$. By Lemma~\ref{lem:refinement} there exists a finite coarse cover $V_1,\ldots,V_n,B$ of $X$ such that $V_i\notclose U_i^c$ for every $i$ and $B\notclose A$. Then $V_1,\ldots,V_n$ are a coarse cover of $A$.

We show the base identity axiom: Let $\phi,\psi\in F^\nu(\closedop{U^c}^c)$ be elements with $\phi|_{\closedop{U_i^c}^c}=\psi|_{\closedop{U_i^c}^c}$ for every $i$. Since $\phi_{V_i}=\psi_{V_i}$ for every $i$ the identity axiom on coarse covers implies $\phi_A=\psi_A$.
  
 Now we show the base gluablity axiom. Let $\phi_i\in F^\nu(\closedop{U_i^c}^c)$ be a section for every $i$ such that $\phi_i|_{\closedop{(U_i\cap U_j)^c}^c}=\phi_j|_{\closedop{(U_i\cap U_j)^c}^c}$ for every $i,j$. Then the $(\phi_i)_{V_i}$ glue to a section $\phi_A$ on $A$ by the gluablity axiom on coarse covers. 
\end{proof}

If $\alpha:F\to F'$ is a morphism of sheaves on $X_{ct}$ then we define for every basic open $\closedop{U^c}^c$:
\begin{align*}
 \alpha^\nu(\closedop{U^c}^c):F^\nu(\closedop{U^c}^c)&\to F'^\nu(\closedop{U^c}^c)\\
 (\varphi_A)_{A\notclose U^c}&\mapsto(\alpha(A)(\varphi_A))_{A\notclose U^c}.
\end{align*}
This definition makes sense since $B\s A$ implies that $(\alpha(A)(\varphi_A))|_B=\alpha(B)(\varphi_A|_B)$. Thus $(\alpha(A)(\varphi_A))_{A\notclose U^c}$ is an element in $\varprojlim_{A\notclose U^c}F'(A)$. By gluing along basic open covers we obtain for every open $\pi\s \nu(X)$ a map $\alpha^\nu(\pi):F^\nu(\pi)\to F'^\nu(\pi)$. We show $\alpha^\nu:F^\nu\to F'^\nu$ is a morphism of sheaves: If $V\s U$ is an inclusion of subsets and $(\varphi_A)_{A\notclose U^c}\in F^\nu(\closedop{U^c}^c)$ an element then
\begin{align*}
 \alpha^\nu(\closedop{V^c}^c)\circ\cdot|_{\closedop{V^c}^c}((\varphi_A)_{A\notclose U^c})
 &=\alpha^\nu(\closedop{V^c}^c)((\varphi_A)_{A\notclose V^c})\\
 &=(\alpha(A)(\varphi_A))_{A\notclose V^c})\\
 &=\cdot|_{\closedop{V^c}^c}((\alpha(A)(\varphi_A))_{A\notclose U^c}))\\
 &=\cdot|_{\closedop{V^c}^c}\circ\alpha^\nu(\closedop{U^c}^c)((\varphi_A)_{A\notclose U^c}).
\end{align*}
Moreover $id_F^\nu=id_{F^\nu}$ and $(\alpha\circ \beta)^\nu=\alpha^\nu\circ\beta^\nu$. Thus we have proved $\cdot^\nu$ is a functor. Namely if $\sheaves{X_{ct}}$ denotes the category of sheaves on $X_{ct}$ and $\sheaves{\nu(X)}$ denotes the category of sheaves on $\nu(X)$ then 
\[
\cdot^\nu:\sheaves{X_{ct}}\to \sheaves{\nu(X)}
\]
is a functor between categories of sheaves.

\begin{lem}
\label{lem:normalprop}
 Let $X$ be a proper metric space. If $U_1^c,\ldots, U_n^c,A\s X$ are subsets with $\closedop{U_1^c}\cap\cdots\cap\closedop{U_n^c}\cap \closedop A=\emptyset$ then there exists a subset $U^c\s X$ with
 \[
  \closedop{(U_1\cap U)^c}^c\cup\cdots\cup\closedop{(U_n\cap U)^c}^c=\closedop{U^c}^c
 \]
 an open cover and $U^c\notclose A$.
\end{lem}
\begin{proof}
 Since $\nu(X)$ is compact there only exists one proximity relation on $\nu(X)$ which induces the topology on $\nu(X)$. Thus the relation $\delta$ defined by $\pi_1\delta\pi_2$ if $\overline{\pi_1}\cap\overline{\pi_2}\not=\emptyset$ and the relation induced by $\close $ on the quotient coincide. Since both $\pi:=\closedop{U_1^c}\cap\cdots\cap\closedop{U_n^c}$ and $\closedop{A}$ are closed sets we obtain
 \[
  \overline\pi\cap \overline{\closedop A}=\pi\cap \closedop A=\emptyset.
 \]
Thus $\pi\notclose \closedop{A}$. Then there exist $U^c,B\s X$ with $\pi\s \closedop{U^c},\closedop A\s \closedop{B}$ and $U^c\notclose B$. This in particular implies that $U^c\notclose A$. Then
\begin{align*}
 \closedop{(U_1\cap U)^c}\cup\cdots\cup\closedop{(U_n\cap U)^c}
 &=\bigcup_i\closedop{U_1^c\cup U^c}^c\\
 &=\left(\bigcap_i \left(\closedop{U_i^c}\cup \closedop {U^c}\right)\right)^c\\
 &=\bigcup_i\closedop{U_i^c}^c\cap \closedop{U^c}^c\\
 &=\pi^c\cap\closedop{U^c}^c\\
 &=\closedop{U^c}^c.
\end{align*}
\end{proof}

Given a subset $A\s X$ of a proper metric the relations $U^c\notclose A$ and $V^c\notclose A$ imply $U^c\cup V^c\notclose A$. Thus $(U^c)_{U^c\notclose A}$ is a directed poset. If $G$ is a sheaf on $\nu(X)$ we define
\[
  \hat G(A):=\varinjlim_{A\notclose U^c}G(\closedop{U^c}^c)
\]
If $A\s B$ then $U^c\notclose B$ implies $U^c\notclose A$. Thus we can define a well defined restriction map $\hat G(B)\to \hat G(A)$ which maps $[\varphi_{U^c}]$ to $[\varphi_{U^c}]$. This makes $\hat G$ a presheaf on $X_{ct}$.

A sheaf $F$ on $X_{ct}$ is called reflective if for every subset $A\s X$ the canonical map 
\[
 \varinjlim_{A\notclose U^c} F(U)\to F(A)
\]
is an isomorphism.

\begin{prop}
\label{prop:catofsh2}
 Let $X$ be a proper metric space. If $G$ is a sheaf on $\nu(X)$ then $\hat G$ is a reflective sheaf on $X_{ct}$.
\end{prop}
\begin{proof}
 Let $A_1,\ldots, A_n$ be a coarse cover of $A\subseteq X$. 
 
 We prove the identity axiom. Let $s\in \hat G(A)$ be a section with $s|_{A_i}=0$ for every $i$. Then there exists $U_i^c\notclose A_i$ with $s_{U_i^c}=t_{U_i^c}$ in $G(\closedop{U_i^c}^c)$. Since $\closedop{A_i}\cap\closedop{U_i^c}=\emptyset$ we obtain
 \begin{align*}
  \closedop{U_1^c}\cap\cdots\cap\closedop{U_n^c}\cap \closedop A
  &\s \closedop{A_1}^c\cap\cdots\cap\closedop{A_n}^c\cap \closedop A\\
  &=(\closedop {A_1}\cup\cdots\cup\closedop{A_n})^c\cap \closedop A\\
  &=\emptyset.
 \end{align*}
By Lemma~\ref{lem:normalprop} there is some $U\subseteq X$ with $\closedop{(U_1\cap U)^c}^c\cup\cdots\cup\closedop{(U_n\cap U)^c}^c= \closedop{U^c}^c$ an open cover and $U^c\notclose A$. Thus the identity axiom for open covers implies $s_{U^c}=0$. This proves $s=0$ in $\hat G(A)$.

Now we prove the gluablity axiom. Let $s_i\in \hat G(A_i)$ be a section for every $i$ with $s_i|_{A_j}=s_j|_{A_i}$ for every $i,j$. Suppose $s_i$ are represented by $(s_i)_{U_i^c}\in G(\closedop{U_i^c}^c)$ with $U_i^c\notclose A_i$. As in the first part of this proof there is some subset $U^c\subseteq X$ with
\[
 \closedop{(U_1\cap U)^c}^c\cup\cdots\cup\closedop{(U_n\cap U)^c}^c=\closedop{U^c}^c
\]
and $U^c\notclose A$. By the gluablity axiom on open covers the $(s_i)_{U_i^c}$ glue to a section $s_{U^c}\in G(\closedop{U^c}^c)$ which represents a section $s\in \hat G(A)$.

Now we show $\hat G$ is reflective. For every $A,U\s X$ with $A\notclose U^c$ there exists $T\s X$ with $A\notclose T^c$ and $T\notclose U^c$. Thus
\[
 \varinjlim_{A\notclose T^c}\hat G(T)=\varinjlim_{A\notclose T^c}\varinjlim_{T\notclose U^c}G(\closedop{U^c}^c)\to \varinjlim_{A\notclose U^c}G(\closedop{U^c}^c)
\]
is an isomorphism.
\end{proof}

If $\beta:G\to G'$ is a morphism of sheaves on $\nu(X)$ and $A\s X$ a subset then 
\begin{align*}
 \hat \beta(A):\hat G(A)&\to \hat G'(A)\\
 [\varphi_{U^c}]&\mapsto [\beta(\closedop{U^c}^c)(\varphi_{U^c})]
\end{align*}
is well defined since $V^c\s U^c$ implies 
\[
 \beta(\closedop{U^c}^c)(\varphi_{V^c}|_{\closedop{U^c}^c})=(\beta(\closedop{V^c}^c))(\varphi_{V^c}))|_{\closedop{U^c}^c}. 
\]
We show $(\hat \beta(A))_{A\s X}$ defines a morphism of sheaves on $X_{ct}$. If $B\s A$ then
\begin{align*}
 \beta(B)\circ\cdot|_B[\varphi_{U^c}]_A
 &=[\beta(\closedop{U^c}^c)(\varphi_{U^c})]_B\\
 &=\cdot|_B\circ\beta(A)[\varphi_{U^c}]_A.
\end{align*}
Moreover $\widehat{id_G}=id_{\hat G}$ and $\widehat{\alpha\circ\beta}=\hat\alpha\circ\hat\beta$. Thus we showed $\hat\cdot $ is a functor between categories of sheaves:
\[
 \hat\cdot:\sheaves{\nu(X)}\to \sheaves{X_{ct}}.
\]

In fact its image is contained in the full subcategory of reflective sheaves.
\begin{thm}
\label{thm:catofsheaves}
If $X$ is a proper metric space the category of reflective sheaves $\sheaves{X_{ct}}$ on $X$ is equivalent to the category of sheaves $\sheaves{\nu(X)}$ on $\nu(X)$ via $\cdot^\nu,\hat\cdot$.
\end{thm}
\begin{proof}
 Let $G$ be a sheaf on $\nu(X)$. Then for every $U\s X$ there is a morphism
 \begin{align*}
  \eta_G(\closedop{U^c}^c):G(\closedop {U^c}^c)&\to (\hat G)^\nu(\closedop{U^c}^c)\\
  s&\mapsto ([s]_U)_{A\notclose U^c}
 \end{align*}
which naturally defines a morphism of sheaves $\eta_G$. We show this map is bijective. Suppose $s\in G(\closedop{U^c}^c)$ is mapped by $\eta_G(\closedop{U^c}^c)$ to $0$. Then for every $A\notclose U^c$ there exists $U_A^c\notclose A$ with $s|_{U_A}=0$. Then
\begin{align*}
 \bigcap_{A\notclose U^c}\closedop{U_A^c}\cap\closedop{U^c}^c 
 &\s \bigcap_{A\notclose U^c}\closedop A^c\cap\closedop{U^c}^c\\
 &=\left(\bigcup_{A\notclose U^c}\closedop A\cup \closedop {U^c}\right)^c\\
 &=(\nu(X))^c\\
 &=\emptyset.
\end{align*}
Thus $(\closedop{U_A^c}^c)_{A\notclose U^c}$ is an open cover of $\closedop{U^c}^c$. The global axiom on open covers of $\nu(X)$ shows that $s=0$ on $\closedop{U^c}^c$. Suppose $([t_{A}]_{U_A})_{A\notclose U^c}$ is an element in $(\hat G)^\nu(\closedop{U^c}^c)$. Then, as before, $(\closedop{U_A^c}^c)_{A\notclose U^c}$ is an open cover of $\closedop{U^c}^c$. Then $(t_A)_{A\notclose U^c}\in \prod_{A\notclose U^c}G(\closedop{U_A^c}^c)$ is an element with
\begin{align*}
 \eta_G(\closedop{U_A^c}^c)(t_A)
 &=([t_A]_{U_A})_{A'\notclose U_A^c}\\
 &=\cdot|_{\closedop{U_A^c}^c}([t_{A'}]_{U_{A'}})_{A'\notclose U^c}
\end{align*}
 Thus $\eta_G$ is surjective. It is easy to see that $\eta:G\mapsto \eta_G$ defines a natural transformation. This way we showed that $\eta$ is a natural isomorphism between $id_{\sheaves{\nu(X)}}$ and $\cdot^\nu\circ\hat\cdot$.
 
 Now let $F$ be a sheaf on $X_{ct}$. Then for every $A\s X$ there is a map 
 \begin{align*}
  \epsilon_F(A):\widehat{(F^\nu)}(A)&\to F(A)\\
  [(s_{A'})_{A'\notclose U^c}]_U&\mapsto s_A.
 \end{align*}
This map is well defined since $[(s_{A'})_{A'\notclose U^c}]=0$ implies there is some $U^c\notclose A$ such that for every $A'\notclose U^c$ the section $s_{A'}=0$ vanishes. This in particular implies that $s_A=0$. Now we show $\epsilon_F(A)$ is injective. If $[(s_{A'})_{A'\notclose U}]_U$ maps to $0$ by $\varepsilon_F(A)$ then the support $\supp((s_{A'})_{A'\notclose U^c})$ is closed in $\nu(X)$. Thus there exists an open $\closedop{V^c}^c\z \closedop{A}$ on which $(s_{A'})_{A'\notclose V^c}$ vanishes. Thus $(s_{A'})_{A'\notclose U^c}$ represents the $0$ element. Now we show $\epsilon_F(A)$ is surjective if $F$ is reflective. If $s_A\in F(A)$ then there exists some $A\notclose U^c$ and $s\in F(U)$ such that $s|_A=s_A$. Then $[(s|_{A'})_{A'\notclose U^c}]_U$ maps by $\epsilon_F(A)$ to $s_A$ 
\end{proof}

\begin{thm}
 If $F$ is a reflective sheaf on $X_{ct}$ then $\check H_{ct}^q(X,F)=\check H^q(\nu(X),F^\nu)$. The right side denotes sheaf cohomology on $\nu(X)$.
\end{thm}
\begin{proof}
We first show if $G$ is a flabby sheaf on $\nu(X)$ then $\hat G$ is a flabby sheaf on $X_{ct}$. For every $U\s X$ the restriction $G(\nu(X))\to G(\closedop{U^c}^c)$ is surjective. If $[s_U]_U\in \varinjlim_{A\notclose U^c} G(\closedop{U^c}^c)$ then there exists some $s_X\in G(\nu(X))=\hat G(X)$ with $s_X|_U=s_U$.

Now we show $\hat\cdot$ is an exact functor. Let
\[
 G_1\xrightarrow{\alpha}G_2\xrightarrow{\beta}G_3
\]
be an exact sequence of sheaves on $\nu(X)$. If $[s_U]\in\ker \hat\beta(A)$ then there exists $A\notclose V^c$ with $\beta(\closedop{V^c}^c(s_U)|_{\closedop{V^c}^c}=0$. Since $\ker\beta\s \im\alpha$ there is a cover $\bigcup\closedop{U_i^c}^c=\closedop{V^c}^c$ and $s_{U_i}\in G_1(\closedop{U_i^c}^c)$ with $\alpha(\closedop{U_i^c}^c)(s_{U_i})=s_V|_{\closedop{U_i^c}^c}$. Then $(\closedop{U_i^c}^c)_i$ cover $\closedop A$. Since $\closedop A$ is compact a finite subcover $\closedop{U_1^c}^c,\cdots, \closedop{U_n^c}^c$ will do. Then $U_1,\cdots,U_n$ form a coarse cover of $A$. By Lemma~\ref{lem:refinement} there exists a coarse cover $V_1,\cdots,V_n$ of $A$ with $V_i\notclose U_i^c$. Then $[s_{U_i}]\in \prod\hat G_1(V_i)$ maps to $[s_U]$ by $\hat \alpha$. Thus we have proved $\ker \hat\beta\s \im\hat \alpha$. If conversely $[\alpha(\closedop{U_i^c}^c)(s_{U_i})]_i\in\prod\hat G_2(A_i)$ with $A_i$ a coarse cover of $A$ and $A_i\notclose U_i^c$ represents an element in $\im\hat \alpha(A)$ then in particular $(\closedop {U_i^c}^c$ is an open cover containing $\closedop A$. By Lemma~\ref{lem:normalprop} there exists $A\notclose U^c$ with $\closedop{U^c}^c$ covered by $\closedop{U_i^c}^c$. Since $\ker\beta\z \im\alpha$ there exists $s_U\in\ker\beta(\closedop{U^c}^)$ with $s_U|_{\closedop{U_i^c}^c}=s_{U_i}$. Then $[s_U]_U\in \ker\hat\beta(A)$ has the property that $[s_U]_U|_{A_i}=\hat\alpha(A_i)([s_{U_i}]_{U_i})$. Thus $\ker\hat\beta\z\im\hat \alpha$. This way we proved $\ker \hat \beta=\im\hat \alpha$, the sequence $\hat G_1\to \hat G_2\to \hat G_3$ is exact at $\hat G_2$.

If $F$ is a reflective sheaf on $X_{ct}$ then there exists a flabby resolution 
 \[
  0\to F^\nu\to G_0\to G_1\to\cdots 
 \]
of sheaves on $\nu(X)$. Since $\hat\cdot$ is an exact functor we obtain an exact resolution of flabby sheaves
\[
 0\to \widehat{F^\nu}\to \hat G_0\to \hat G_1\to \cdots
\]
with an isomorphism $\widehat{F^\nu}\to F$. The global section functor on the reduced sequences gives the same result.
\end{proof}

\begin{prop}
 If $A$ is an abelian group and $X$ a proper metric space then $A_X\!^\nu=A_{\nu(X)}$ on $\nu(X)$ and $\hat A_{\nu(X)}=A_X$ on $X_{ct}$ are isomorpic. In particular $A_X$ is a reflective sheaf.                                                                     
\end{prop}
\begin{proof}
If $B\s X$ is a subset and $i:B\to X$ denotes the inclusion then $\nu(i):\nu(B)\to \nu(X)$ is an inclusion of a closed subset. We have $\hat A_{\nu(X)}(B)=\varinjlim_{B\notclose U^c} A_{\nu(X)}(\closedop{U^c}^c)=\nu(i)^{-1}A_{\nu(X)}(\closedop B)=A_{\nu(B)}(\nu(B))$. There is a bijective map $A_X(B)\to A_{\nu(B)}(\nu(B))$ defined as follows: A section in $A_X(B)$ is represented by a Higson function $\varphi:B\to A$ where $A$ is equipped with the word length metric. This function can be extended to the boundary $\nu(B)$ since it is Higson. Then this function is a continuous map $\nu(B)\to A$ where $A$ is equipped with the discrete topology. Since a coarse disjoint union of $B$ is 1:1 with disjoint unions of $\nu(B)$ by clopen sets we obtain a bijection. This tells us $A_X=\hat A_{\nu(X)}$. Applying the $\cdot^\nu$ functor we obtain a bijection $A_X^\nu=\hat A_{\nu(X)}^\nu=A_{\nu(X)}$.
\end{proof}

\begin{thm}
\label{thm:asdimvanish}
 If $X$ is a proper metric space with $\asdim(X)=n$ then $\check H^q_{ct}(X,\mathcal F)=0$ for every reflective sheaf $\mathcal F$ and $q>n$.
\end{thm}
\begin{proof}
 The space $\nu(X)$ is paracompact since it is compact. By \cite[Chapitre~II.5.12]{Godement1958} it is sufficient to show that the covering dimension of $\nu(X)$ does not exceed $n$. By\cite[Theorem~1.1]{Dranishnikov1998} we obtain $\dim(\nu X)\le \asdim(X)$. Thus the result follows.
\end{proof}

This result can be used to finalize our computations in Chaper~\ref{sec:computations}.

\begin{thm}
\label{thm:hqztree}
 If $A$ is a finite abelian group then 
 \[
  \check H_{ct}^q(\Z_{\ge0};A)=\begin{cases}
                              A & q=0\\
                              0 &\mbox{otherwise.}
                             \end{cases}
 \]
If $F_n$ denotes the free group with $n<\infty$ generators and $A$ is finite again then
\[
 \check H_{ct}^q(F_n;A)=\begin{cases}
                         \bigoplus_\N A & q=0\\
                         0 & \mbox{otherwise.}
                        \end{cases}
\]
\end{thm}
\begin{proof}
 The cohomology in degree $0$ is clear since $\Z_{\ge0}$ has one end and $F_n$ has infinitely many ends. In degree $1$ cohomology with finite coefficients vanishes by Lemma~\ref{lem:h1z} and Theorem~\ref{thm:h1tree}, respectively. Now both $\Z_{\ge 0}$ and trees have asymptotic dimension $1$ \cite{Nowak2012}. Then Theorem~\ref{thm:asdimvanish} implies the higher cohomology groups vanish.
\end{proof}

\section{Coarse homotopy invariance}
This chapter proves Theorem~\ref{thm:introHomotopy}.
\begin{lem}
\label{lem:immediate}
 If $\alpha_0,\ldots,\alpha_q:X\to Y$ are coarse maps which are close to each other and $\varphi\in CX_b^q(Y,A)$ is a cochain then
 \begin{enumerate}
  \item $\varphi\circ (\alpha_0\times\cdots\times \alpha_q)$ is a cochain;
  \item the composition of $(y_0,\ldots,y_{q-1})\mapsto (y_0,\ldots,y_i,y_i,\ldots,y_{q-1})$ with the map $\varphi$ is a cochain in $CX_b^{q-1}(Y,A)$;
  \item the composition of $(y_0,\ldots,y_{q-2})\mapsto (y_0,\ldots,y_i,y_i,\ldots,y_j,y_j,\ldots,y_{q-2})$ with $\varphi$ is a cochain in $CX_b^{q-2}(Y,A)$.
 \end{enumerate}
\end{lem}
\begin{proof}
 We prove 1. first.  Suppose $D\s Y^{q+1}$ is cocontrolled. We show $(\alpha_0\times\cdots\times\alpha_q)^{-1}(D)$ is cocontrolled. If $R\ge 0$ then $(\alpha_0\times\cdots\times \alpha_q)(\Delta_R)$ is cocontrolled since $\alpha_0,\ldots,\alpha_q$ are coarse and close to each other. Thus there exists some $0\in Y$ and $S\ge 0$ with
 \[
  D\cap (\alpha_0\times\cdots\times \alpha_q)(\Delta_R)\s (\Delta_S[0])^{q+1}
 \]
Then
\begin{align*}
 (\alpha_0\times\cdots\times \alpha_q)^{-1}(D)\cap \Delta_R
 &\s (\alpha_0\times\cdots\times \alpha_q)^{-1}(D)\cap (\alpha_0\times\cdots\times \alpha_q)^{-1}\circ (\alpha_0\times\cdots\times \alpha_q)(\Delta_R)\\
 &\s (\alpha_0\times\cdots\times \alpha_q)^{-1}(\Delta_S[0]^{q+1})
\end{align*}
is bounded. Thus $(\alpha_0\times\cdots\times \alpha_q)^{-1}(D)$ is cocontrolled. Now $\supp\varphi$ is cocontrolled. Then
\begin{align*}
 \supp (\varphi\circ (\alpha_0\times\cdots\times \alpha_q))
 &= (\alpha_0\times\cdots\times \alpha_q)^{-1}(\supp \varphi|_{\im(\alpha_0\times\cdots\times \alpha_q)})\\
 &=(\alpha_0\times\cdots\times \alpha_q)^{-1}(\supp\varphi)
\end{align*}
is cocontrolled. If $\varphi\in C_{A_1,\ldots,A_n}(Y,A)$ then $\varphi\circ (\alpha_0\times\cdots\times \alpha_q)\in C_{\alpha_0^{-1}(A_{i_0})\cap\cdots\cap \alpha_q^{-1}(A_{i_q})}(X,A)$. Thus $\varphi\circ(\alpha_0\times\cdots\times \alpha_q)$ is blocky. This way we showed that $\varphi\circ(\alpha_0\times\cdots\times \alpha_q)$ is a cochain. 

Now we prove 2. Name the map $(y_0,\ldots,y_{q-1})\mapsto (y_0,\ldots,y_i,y_i,\ldots,y_{q-1})$ by $\delta_i$. If $R\ge 0$ then there exist $0\in Y, S\ge0$ such that $\supp \varphi|_{\Delta_R}\s (\Delta_S[0])^{q+1}$. Now $\delta_i(\Delta_R^{q-1})\s \Delta_R^q$ and $\delta_i^{-1}((\Delta_S[0])^{q+1})\s (\Delta_S[0])^q$. We use this to prove
\begin{align*}
 \supp(\varphi\circ \delta_i)|_{\Delta_R^{q-1}}
 &=\delta_i^{-1}(\supp\varphi|_{\delta_i(\Delta_R^{q-1})})\\
 &\s \delta_i^{-1}(\supp\varphi|_{\Delta_R^q})\\
 &\s \delta_i^{-1}((\Delta_S[0])^{q+1})\\
 &\s (\Delta_S[0])^q.
 \end{align*}
Thus $\varphi\circ\delta_i$ has cocontrolled support. Moreover if $\varphi\in C^q_{A_1,\ldots,A_n}(Y,A)$ then we have $\varphi\circ\delta_i\in C^{q-1}_{A_1,\ldots,A_n}(Y,A)$. Thus $\varphi \circ\delta_i$ is blocky. This way we showed $\varphi\circ \delta_i$ defines a cochain.

The proof of 3. is similar to the proof of 2. and left to the reader.
\end{proof}

\begin{lem}
\label{lem:technicallemma1}
 If $I$ is a metric space, $h_0,h_1,h_2:I\to I$ coarse maps which are close to each other and $\varphi\in CX_b^q(I,A)$ is a cocycle then 
 \begin{align*}
  \sum_{i=0}^{q-1}(-1)^i( &\varphi(h_0(z_0),\ldots,h_0(z_i),h_1(z_i),\ldots,h_1(z_{q-1}))+\varphi(h_1(z_0),\ldots,h_1(z_i),h_2(z_i),\ldots,h_2(z_{q-1}))\\
  &-\varphi(h_0(z_0),\ldots,h_0(z_i),h_2(z_i),\ldots,h_2(z_{q-1})))
 \end{align*}
 is a coboundary in $CX_b^{q-1}(I,A)$.
\end{lem}
\begin{proof}
First we define for $0\le i\le j\le q-2$ a map
\begin{align*}
 \chi_{ij}:I^{q-1}&\to A\\
 (x_0,\ldots,x_{q-2})&\mapsto \varphi(h_0(x_0),\ldots,h_0(x_i),h_1(x_i),\ldots,h_1(x_j),h_2(x_j),\ldots,h_2(x_{q-2}))
\end{align*}
This map defines a cochain by Lemma~\ref{lem:immediate}. 

If $(z_0,\ldots,z_{k-1})\in I^q$ then $(\hat z_k)$ is short for $(z_0,\ldots,\hat z_k,\ldots,z_{k-1})\in I^{q-1}$. We compute 
 \begin{align*}
  0
  &=\sum_{0\le i\le j\le q-1}(-1)^{i+j}d_q\varphi(h_0(z_0),\ldots,h_0(z_i),h_1(z_i),\ldots,h_1(z_j),h_2(z_j),\ldots,h_2(z_{q-1}))\\
  &=\sum_{0\le i\le j\le q-1}\\
  &(\sum_{k=0}^i (-1)^{i+j+k}\varphi(h_0(z_0),\ldots,\widehat{h_0(z_k)},\ldots,h_0(z_i),h_1(z_i),\ldots,h_1(z_j),h_2(z_j),\ldots,h_2(z_{q-1}))\\
  &+\sum_{k=i}^j(-1)^{i+j+k+1}\varphi(h_0(z_0),\ldots,h_0(z_i),h_1(z_i),\ldots,\widehat{h_1(z_k)},\ldots,h_1(z_j),h_2(z_j),\ldots,h_2(z_{q-1}))\\
  &+\sum_{k=j}^{q-1}(-1)^{i+j}\varphi(h_0(z_0),\ldots,h_0(z_i),h_1(z_i),\ldots,h_1(z_j),h_2(z_j),\ldots,\widehat{h_2(z_k)},\ldots,h_2(z_{q-1})))\\
  &=\sum_{0\le i\le j\le q-1}\\
  &(\sum_{k=0}^{i-1}(-1)^{i+j+k}\chi_{i-1,j-1}(\hat z_k)\\
  &+(-1)^{i+j+i}\varphi(h_0(z_0),\ldots,h_0(z_{i-1}),h_1(z_i),\ldots,h_1(z_j),h_2(z_j),\ldots,h_2(z_{q-1}))^{\boxed{$$1$$}}\\
  &+(-1)^{i+j+i+1}\varphi(h_0(z_0),\ldots,h_0(z_i),h_1(z_{i+1}),\ldots,h_1(z_j),h_2(z_j),\ldots,h_2(z_{q-1}))^{\boxed{$$2$$}}\\
  &+\sum_{k=i+1}^{j-1}(-1)^{i+j+k+1}\chi_{i,j-1}(\hat z_k)\\
  &+(-1)^{i+j+j+1}\varphi(h_0(z_0),\ldots,h_0(z_i),h_1(z_i)\ldots,h_1(z_{j-1}),h_2(z_j),\ldots,h_2(z_{q-1}))^{\boxed{$$3$$}}\\
  &+(-1)^{i+j+j} \varphi(h_0(z_0),\ldots,h_0(z_i),h_1(z_i),\ldots,h_1(z_j),h_2(z_{j+1}),\ldots,h_2(z_{q-1}))^{\boxed{$$4$$}}\\
  &+\sum_{k=j+1}^{q-1}(-1)^{i+j+k}\chi_{ij}(\hat z_k))
 \end{align*}
We arrived at a sum where the terms marked with ${\boxed{$$1$$}},{\boxed{$$2$$}},{\boxed{$$3$$}},{\boxed{$$4$$}}$ either contribute the desired terms or cancel each other out. The terms with $\chi_{ij}$ add to a coboundary.

We first look at the terms ${\boxed{$$1$$}},{\boxed{$$2$$}},{\boxed{$$3$$}},{\boxed{$$4$$}}$. If $0\le i<j\le q-1$ the term $\boxed{$$1$$}$ for $i+1,j$ cancels with the term $\boxed{$$2$$}$ for $i,j$. If $0\le i<j\le q-1$ the term $\boxed{$$3$$}$ for $i,j$ cancels with the term $\boxed{$$4$$}$ for $i,j-1$. We did not yet count the terms $\boxed{$$1$$}$ for $i=0,0\le j\le{q-1}$ which give $(-1)^j\varphi(h_1(z_0),\ldots,h_1(z_j),h_2(z_j),\ldots,h_2(z_{q-1}))$. The terms $\boxed{$$4$$}$ for $j=q-1,0\le i\le q-1$ contribute $(-1)^i\varphi(h_0(z_0),\ldots,h_0(z_i),h_1(z_i),\ldots,h_1(z_{q-1}))$. Finally if $0\le i=j\le q-1$ then $\boxed{$$2$$}=\boxed{$$3$$}$ is counted only once and contributes $(-1)^{i+i+i+1} \varphi(h_0(z_0),\ldots,h_0(z_i),h_2(z_i),\ldots,h_2(z_{q-1}))$.

It remains to show that the other terms contribute a coboundary:
\begin{align*}
\cdots
 &=\sum_{0\le i\le j\le q-1}
 (\sum_{k=0}^{i-1}(-1)^{i+j+k}\chi_{i-1,j-1}(\hat z_k)
 +\sum_{k=i+1}^{j-1}(-1)^{i+j+k+1}\chi_{i,j-1}(\hat z_k)\\
 &+\sum_{k=j+1}^{q-1}(-1)^{i+j+k}\chi_{ij}(\hat z_k))\\
 &=\sum_{0\le i\le j\le q-2}(\sum_{k=0}^i(-1)^{i+j+k}\chi_{i,j}(\hat z_k) +\sum_{k=i+1}^{j}(-1)^{i+j+k}\chi_{i,j}(\hat z_k)\\
 &+\sum_{k=j+1}^{q-1}(-1)^{i+j+k}\chi_{ij}(\hat z_k))\\
 &=d_{q-2}(\sum_{0\le i\le j\le q-2}(-1)^{i+j}\chi_{ij})(z_0,\ldots,z_{q-1}).
\end{align*}

\end{proof}

\begin{lem}
\label{lem:technicallemma2}
 If $I$ is a metric space, $\varphi\in CX_b^q(I,A)$ a cochain and $(h_t)_t$ a family of coarse maps $I\to I$ with the properties
 \begin{enumerate}
  \item $d(z_0,z_1)\ge d(h_t(z_0),h_t(z_1))$ for every $z_0,z_1\in I,t$;
  \item $d(z_0,0)=d(h_t(z_0),0)$ for every $z_0\in I,t$ and some $0\in I$;
 \end{enumerate}
 then for every $R\ge0$ there exists $S\ge 0$ such that
 \[
  \supp \varphi(h_{t_0}(\cdot),\ldots,h_{t_q}(\cdot))|_{\Delta^q_R}\s (\Delta_S[0])^{q+1}.
 \]
 independent of $t_0,\ldots,t_q$.
\end{lem}
\begin{proof}
 If $R\ge 0$ then $\supp \varphi|_{\Delta_R}$ is bounded, namely contained in $(\Delta_S[0])^{q+1}$ for some $S\ge 0$. 
\begin{align*}
 \supp (\varphi\circ (h_{t_0}\times\cdots\times h_{t_q}))|_{\Delta^q_R}
 &= (h_{t_0}\times \cdots \times h_{t_q})^{-1}(\supp \varphi|_{(h_{t_0}\times \cdots \times h_{t_q})(\Delta_R^q)})\\
 &\s (h_{t_0}\times \cdots \times h_{t_q})^{-1}(\supp \varphi|_{\Delta_R^q})\\
 &\s (h_{t_0}\times \cdots \times h_{t_q})^{-1}((\Delta_S[0])^{q+1})\\
 &= (\Delta_S[0])^{q+1}.
\end{align*}
\end{proof}

The proof of Theorem~\ref{thm:homotopyinvariance} can be illustrated by an example. Proposition~\ref{prop:intervalacyclic} carries out the essential step of the proof for $X=\Z_{\ge 0}$. 
\begin{prop}
\label{prop:intervalacyclic}
If $I_0:=\Z_{\ge0}\times\Z_{\ge0}\s \Z^2$ the projection
\begin{align*}
 \pi:I_0&\to \Z_{\ge 0}\\
 (x,y)&\mapsto x+y
\end{align*}
induces an isomorphism in cohomology $\pi^*$ inverse to the induced map $\iota^*$ associated to the inclusion $\iota:x\mapsto (x,0)$.
\end{prop}
\begin{proof}
In the following proof $z_i$ is short for $\begin{pmatrix}x_i\\y_i\end{pmatrix}$ and $\bar z$ abbreviates $\left(\begin{pmatrix}x_0\\y_0\end{pmatrix},\ldots,\begin{pmatrix}x_q\\y_q\end{pmatrix}\right)$ or $\left(\begin{pmatrix}x_0\\y_0\end{pmatrix},\ldots,\begin{pmatrix}x_{q-1}\\y_{q-1}\end{pmatrix}\right)$.

We show the map 
\begin{align*}
 p:=\iota\circ\pi:I_0&\to I_0\\
 (x,y)&\mapsto (x+y,0)
\end{align*}
induces the same map $p^*$ in cohomology as the identity on $I_0$.

For $t\in \N_0$ we define an auxilary map 
\begin{align*}
h_t:I_0&\to I_0\\
(x,y)&\mapsto\begin{cases}
              (x+t,y-t) & t<y\\
              (x+y,0) & t\ge y.
             \end{cases}
\end{align*}
We obtain $p(x,y)=h_y(x,y)$. The $(h_t)_t$ satisfy the conditions of Lemma~\ref{lem:technicallemma1} and Lemma~\ref{lem:technicallemma2}.

Suppose $q\ge 2$. Let $\varphi\in CX^q_b(I_0,A)$ be a cocyle. Then $(h_t^*-id_{I_0}^*)\varphi\in \im d_{q-1}$ since $h_t,id_{I_0}$ are close. Thus there is some $\psi_t\in CX_b^{q-1}(I_0,A)$ with $d_{q-1}\psi_t=h_t^*\varphi-\varphi$. Namely 
\[
 \psi_t(z_0,\ldots,z_{q-1})=\sum_{i=0}^{q-1}(-1)^i\varphi(z_0,\ldots,z_i,h_t(z_i),\ldots,h_t(z_{q-1})).
\]
Then define the map
\[
\tilde\psi\left(\begin{pmatrix}x_0\\y_0\end{pmatrix},\ldots,\begin{pmatrix}x_{q-1}\\y_{q-1}\end{pmatrix}\right)
=\sum_{t=0}^{\max(y_j)-1}(h_t^*\psi_1)\left(\begin{pmatrix}x_0\\y_0\end{pmatrix},\ldots,\begin{pmatrix}x_{q-1}\\y_{q-1}\end{pmatrix}\right).
\]
This map is well-defined since for each fixed point in $I_0^q$, only finitely many terms in the above sum are defined. If $R\ge0$ then Lemma~\ref{lem:technicallemma2} implies that there exists some $S\ge0$ such that each summand of $\psi|_{\Delta_R^{q-1}}$ has support contained in $(\Delta_S[0])^q$. This implies that $\supp (\tilde\psi|_{\Delta_R^{q-1}})\s (\Delta_S[0])^q$. Thus $\tilde\psi$ has cocontrolled support. Now $\tilde \psi$ may or may not be blocky. We have to go the extra step to produce a map with cocontrolled support which is also blocky. To obtain such a map we are going to add a coboundary.

By Lemma~\ref{lem:technicallemma1} we obtain 
\begin{align*}
 h_0^*\psi_1(\bar z)+h_1^*\psi_1(\bar z)
 &=\sum_{i=0}^{q-1}(-1)^i\varphi(h_0(z_0),\ldots,h_0(z_i),h_1(z_i),\ldots,h_1(z_{q-1}))\\
 &+\sum_{i=0}^{q-1}(-1)^i\varphi(h_1(z_0),\ldots,h_1(z_i),h_2(z_i),\ldots,h_2(z_{q-1}))\\
 &=\sum_{i=0}^{q-1}(-1)^i\varphi(h_0(z_0),\ldots,h_0(z_i),h_2(z_i),\ldots,h_2(z_{q-1}))+d_{q-2}\chi_1(\bar z)\\
 &=\psi_2(\bar z)+d_{q-2}\chi_1(\bar z).
\end{align*}
in the next step we obtain
\begin{align*}
 \psi_2(\bar z)+h_2^*\psi_1(\bar z)
 &=\sum_{i=0}^{q-1}(-1)^i\varphi(h_0(z_0),\ldots,h_0(z_i),h_2(z_i),\ldots,h_2(z_{q-1}))\\
 &+\sum_{i=0}^{q-1}(-1)^i\varphi(h_2(z_0),\ldots,h_2(z_i),h_3(z_i),\ldots,h_3(z_{q-1}))\\
 &=\sum_{i=0}^{q-1}(-1)^i\varphi(h_0(z_0),\ldots,h_0(z_i),h_3(z_i),\ldots,h_3(z_{q-1}))+d_{q-2}\chi_2(\bar z)\\
 &=\psi_3(\bar z)+d_{q-2}\chi_2(\bar z).
\end{align*}
Successively we obtain
\[
\tilde\psi(\bar z)
=\psi_{\max(y_j)}\left(\begin{pmatrix}x_0\\y_0\end{pmatrix},\ldots,\begin{pmatrix}x_{q-1}\\y_{q-1}\end{pmatrix}\right)+\sum_{t=1}^{\max(y_j)-1}d_{q-2}\chi_t\left(\begin{pmatrix}x_0\\y_0\end{pmatrix},\ldots,\begin{pmatrix}x_{q-1}\\y_{q-1}\end{pmatrix}\right)
\]
By the proof of Lemma~\ref{lem:technicallemma1} the map $d_{q-2}\chi_t$ satisfies the conditions of Lemma~\ref{lem:technicallemma2}. Thus the sum $\sum_{t=1}^{\max(y_j)-1}d_{q-2}\chi_t$ has cocontrolled support. This implies that the map
\[
 \psi\left(\begin{pmatrix}x_0\\y_0\end{pmatrix},\ldots,\begin{pmatrix}x_{q-1}\\y_{q-1}\end{pmatrix}\right)
 :=\psi_{\max(y_i)}\left(\begin{pmatrix}x_0\\y_0\end{pmatrix},\ldots,\begin{pmatrix}x_{q-1}\\y_{q-1}\end{pmatrix}\right)
\]
has cocontrolled support. We have
\[
 \psi(z_0,\ldots,z_{q-1})=\sum_{i=0}^{q-1}(-1)^i\varphi(z_0,\ldots,z_i,p(z_i),\ldots,p(z_{q-1})).
\]
Thus $\psi$ is blocky, namely if $\varphi\in C^q_{A_1,\ldots, A_n}(I_0,A)$ then $\psi\in C^{q-1}_{p^{-1}A_i\cap A_j}(I_0,A)$. This way we have proved that $\psi$ is a cochain.

Lastly we have
\begin{align*}
 (d_{q-1}\psi)(\bar z)
&=\psi\left(\begin{pmatrix}x_1\\y_1\end{pmatrix},\ldots,\begin{pmatrix}x_q\\y_q\end{pmatrix}\right)-\cdots\pm \psi\left(\begin{pmatrix}x_0\\y_0\end{pmatrix},\ldots,\begin{pmatrix}x_{q-1}\\y_{q-1}\end{pmatrix}\right)\\
&=\psi_{\max_iy_i}\left(\begin{pmatrix}x_1\\y_1\end{pmatrix},\ldots,\begin{pmatrix}x_q\\y_q\end{pmatrix}\right)-\cdots\pm \psi_{\max_iy_i}\left(\begin{pmatrix}x_0\\y_0\end{pmatrix},\ldots,\begin{pmatrix}x_{q-1}\\y_{q-1}\end{pmatrix}\right)\\
&=(d_{q-1}\psi_{\max_iy_i})(\bar z)\\
&=h_{\max_i y_i}^*\varphi(\bar z)-\varphi(\bar z)\\
&=(p^*\varphi-\varphi)(\bar z)
 \end{align*}
Thus $p^*\varphi-\varphi$ defines a coboundary. This way we showed $p^*$ induces the identity on $\check H_{ct}^q(I_0,A)$ for $q\ge 1$. It remains to show the statement for $q=0$.

Since $I_0$ is one-ended a cocycle $\varphi\in CY^0(I_0,A)$ is represented by a constant $a\in A$ function on $I_0$ except on a bounded set. Then $p^*\varphi$ is constant $a\in A$ except on a bounded set. Thus $p^*$ is the same map as the identity on $HY^0(I_0,A)$.
\end{proof}

Now denote $I:=\R_{\ge 0}\times \R_{\ge 0}$ and equip this space with the Manhattan metric. Namely if $(s_1,t_1),(s_2,t_2)\in I$ then
\[
 d((s_1,t_1),(s_2,t_2))=|s_2-s_1|+|t_2-t_1|.
\]
If $X$ is a metric space and $x_0\in X$ a point then the \emph{asymptotic product of $X$ and $I$} is defined to be
\[
 X\ast I=\{(x,i)\in X\times I\mid d(x,x_0)=d(i,(0,0))\}
\]
The paper \cite{Hartmann2019b} shows that $X\ast I$ is the pullback of $d(\cdot,x_0)$ and $d(\cdot,(0,0))$. Moreover we can define a well-defined homotopy theory: If $X$ is a metric space define maps
\begin{align*}
\begin{aligned}
 \iota_0:X&\to X\ast I\\
 x&\mapsto (x,(d(x,x_0),0))
 \end{aligned}
 &
 \begin{aligned}
 \iota_1:X&\to X\ast I\\
 x&\mapsto (x,(0,d(x,x_0))) 
 \end{aligned}
\end{align*}

\begin{defn}
 If $\alpha,\beta:X\to Y$ are two coarse maps they are \emph{coarsely homotopic} if there exists a coarse map $h:X\ast I\to Y$ with $h\circ \iota_0=\alpha$ and $h\circ\iota_1=\beta$.
\end{defn}

The paper \cite{Hartmann2019b} shows coarse homotopy is an equivalence relation and compares this theory with other homotopy theories on the coarse category. 

\begin{lem}
\label{lem:0degreestep}
 If $X$ is a metric space then $e(X\ast I)=e(X)$.
\end{lem}
\begin{proof}
 Denote by $\pi$ the projection of $X\ast I$ to the first factor. Since $\pi$ is a surjective coarse map, the inequality $e(X\ast I)\ge e(X)$ follows easily.
 
 Now suppose $X\ast I=A\sqcup B$ is a coarse disjoint union. This means $A,B$ are disjoint and form a coarse cover of $X$. Namely the set $\Delta_R[A]\cap \Delta_R[B]$ is bounded for every $R\ge 0$. Then the set $\{x\in X\mid\exists i,j\in I:(x,i)\in A,(x,j)\in B\}$ is bounded. Without loss of generality we assume it is empty. Thus for fixed $x\in X$ either $(x,i)\in A$ for every $i\in I$ or $(x,i)\in B$ for every $i\in I$. 
 
 Now we show $\pi(A),\pi(B)$ form a coarse disjoint union. They are disjoint by the assumption. Let $R\ge 0$ be a number and let $x\in \pi(A),y\in \pi(B)$ be two points with $d(x,y)\le R$. Then $(x,0)\in A,(y,0)\in B$ with $d((x,0),(y,0))\le R$. Then the set $((x,0),(y,0))$ is bounded which implies $(x,y)$ is bounded.
\end{proof}

\begin{thm}
\label{thm:homotopyinvariance}
 If two maps $\alpha,\beta:X\to Y$ are coarsely homotopic then they induce the same map in cohomology.
\end{thm}
\begin{proof}
 We just need to show that the projection $\pi:X\ast I\to X$ which sends an element $(x,i)$ to $x$ induces an isomorphism in cohomology. Indeed since $\pi\circ\iota_0=id_X=\pi\circ \iota_1$ the maps $\iota_0^*,\iota_1^*$ are both the unique inverse to $\pi^*$. Then $\alpha=h\circ\iota_0$ and $\beta=h\circ\iota_1$ induce the same map in cohomology.
 
 Now Proposition~\ref{prop:intervalacyclic} already showed $\pi^*$ is an isomorphism with $\Z_{\ge0}$ in place of $X$ and $I_0$ in place of $I$. The same proof can be transfered to this situation where we use Lemma~\ref{lem:0degreestep} for the step in degree $0$.
 
 Namely we proceed as follows. In the following proof $z_i$ is short for $(x_i,(s_i,t_i))$ and $\bar z$ abbreviates $(z_0,\ldots,z_q)$ or $(z_0,\ldots,z_{q-1})$.

We show the map 
\begin{align*}
 p:=\iota_0\circ\pi:X\ast I&\to X\ast I\\
 (x,(s,t))&\mapsto (x,(s+t,0))
\end{align*}
induces the same map $p^*$ in cohomology as the identity on $X\ast I$.

For $n\in \N_0$ we define an auxilary map 
\begin{align*}
h_n:X\ast I&\to X\ast I\\
(x,(s,t))&\mapsto\begin{cases}
              (x,(s+n,t-n)) & n<t\\
              (x,(s+t,0)) & n\ge t.
             \end{cases}
\end{align*}
We obtain $p(x,(s,t))=(x,(s+t,0))=h_{\lfloor t\rfloor+1}(x,(s,t))$. The $(h_n)_n$ satisfy the conditions of Lemma~\ref{lem:technicallemma1} and Lemma~\ref{lem:technicallemma2}.

Suppose $q\ge 2$. Let $\varphi\in CX^q_b(X\ast I,A)$ be a cocyle. Then $(h_n^*-id_{X\ast I}^*)\varphi\in \im d_{q-1}$ since $h_n,id_{X\ast I}$ are close. Thus there is some $\psi_n\in CX_b^{q-1}(X\ast I,A)$ with $d_{q-1}\psi_n=h_n^*\varphi-\varphi$. Namely 
\[
 \psi_n(z_0,\ldots,z_{q-1})=\sum_{i=0}^{q-1}(-1)^i\varphi(z_0,\ldots,z_i,h_n(z_i),\ldots,h_n(z_{q-1})).
\]
Then define the map
\[
\tilde\psi((x_0,(s_0,t_0)),\ldots,(x_{q-1},(s_{q-1},t_{q-1})))
=\sum_{n=0}^{\lfloor\max(t_j)\rfloor}(h_n^*\psi_1)((x_0,(s_0,t_0)),\ldots,(x_{q-1},(s_{q-1},t_{q-1}))).
\]
This map is well-defined since for each fixed point in $(X\ast I)^q$, only finitely many terms in the above sum are defined. If $R\ge0$ then Lemma~\ref{lem:technicallemma2} implies that there exists some $S\ge0$ such that each summand of $\psi|_{\Delta_R^{q-1}}$ has support contained in $(\Delta_S[0])^q$. This implies that $\supp (\tilde\psi|_{\Delta_R^{q-1}})\s (\Delta_S[0])^q$. Thus $\tilde\psi$ has cocontrolled support. Now $\tilde \psi$ may or may not be blocky. We have to go the extra step to produce a map with cocontrolled support which is also blocky. To obtain such a map we are going to add a coboundary.

By Lemma~\ref{lem:technicallemma1} we obtain 
\begin{align*}
 h_0^*\psi_1(\bar z)+h_1^*\psi_1(\bar z)
 &=\sum_{i=0}^{q-1}(-1)^i\varphi(h_0(z_0),\ldots,h_0(z_i),h_1(z_i),\ldots,h_1(z_{q-1}))\\
 &+\sum_{i=0}^{q-1}(-1)^i\varphi(h_1(z_0),\ldots,h_1(z_i),h_2(z_i),\ldots,h_2(z_{q-1}))\\
 &=\sum_{i=0}^{q-1}(-1)^i\varphi(h_0(z_0),\ldots,h_0(z_i),h_2(z_i),\ldots,h_2(z_{q-1}))+d_{q-2}\chi_1(\bar z)\\
 &=\psi_2(\bar z)+d_{q-2}\chi_1(\bar z).
\end{align*}
in the next step we obtain
\begin{align*}
 \psi_2(\bar z)+h_2^*\psi_1(\bar z)
 &=\sum_{i=0}^{q-1}(-1)^i\varphi(h_0(z_0),\ldots,h_0(z_i),h_2(z_i),\ldots,h_2(z_{q-1}))\\
 &+\sum_{i=0}^{q-1}(-1)^i\varphi(h_2(z_0),\ldots,h_2(z_i),h_3(z_i),\ldots,h_3(z_{q-1}))\\
 &=\sum_{i=0}^{q-1}(-1)^i\varphi(h_0(z_0),\ldots,h_0(z_i),h_3(z_i),\ldots,h_3(z_{q-1}))+d_{q-2}\chi_2(\bar z)\\
 &=\psi_3(\bar z)+d_{q-2}\chi_2(\bar z).
\end{align*}
Successively we obtain
\begin{align*}
\tilde\psi(\bar z)
&=\psi_{\lfloor\max(t_j)\rfloor+1}((x_0,(s_0,t_0)),\ldots,(x_{q-1},(s_{q-1},t_{q-1})))\\
&+\sum_{n=1}^{\lfloor\max(t_j)\rfloor}d_{q-2}\chi_n((x_0,(s_0,t_0)),\ldots,(x_{q-1},(s_{q-1},t_{q-1}))
\end{align*}
By the proof of Lemma~\ref{lem:technicallemma1} the map $d_{q-2}\chi_t$ satisfies the conditions of Lemma~\ref{lem:technicallemma2}. Thus the sum $\sum_{t=1}^{\max(y_j)-1}d_{q-2}\chi_t$ has cocontrolled support. This implies that the map
\[
 \psi((x_0,(s_0,t_0)),\ldots,(x_{q-1},(s_{q-1},t_{q-1})))
 :=\psi_{\lfloor\max(t_j)\rfloor+1}((x_0,(s_0,t_0)),\ldots,(x_{q-1},(s_{q-1},t_{q-1})))
\]
has cocontrolled support. We have
\[
 \psi(z_0,\ldots,z_{q-1})=\sum_{i=0}^{q-1}(-1)^i\varphi(z_0,\ldots,z_i,p(z_i),\ldots,p(z_{q-1})).
\]
Thus $\psi$ is blocky, namely if $\varphi\in C^q_{A_1,\ldots, A_n}(X\ast I,A)$ then $\psi\in C^{q-1}_{p^{-1}A_i\cap A_j}(X\ast I,A)$. This way we have proved that $\psi$ is a cochain.

Lastly we have
\begin{align*}
 (d_{q-1}\psi)(\bar z)
&=\psi((x_1,(s_1,t_1)),\ldots,(x_q,(s_q,t_q)))-\cdots\pm \psi((x_0,(s_0,t_0)),\ldots,(x_{q-1},(s_{q-1},t_{q-1})))\\
&=\psi_{\lfloor \max(t_j)\rfloor+1}((x_1,(s_1,t_1)),\ldots,(x_q,(s_q,t_q)))-\cdots\\
&\pm \psi_{\lfloor \max(t_j)\rfloor+1}((x_0,(s_0,t_0)),\ldots,(x_{q-1},(s_{q-1},t_{q-1})))\\
&=(d_{q-1}\psi_{\lfloor \max(t_j)\rfloor+1})(\bar z)\\
&=h_{\lfloor \max(t_j)\rfloor+1}^*\varphi(\bar z)-\varphi(\bar z)\\
&=(p^*\varphi-\varphi)(\bar z)
 \end{align*}
Thus $p^*\varphi-\varphi$ defines a coboundary. This way we showed $p^*$ induces the identity on $\check H_{ct}^q(X\ast I,A)$ for $q\ge 1$. It remains to show the statement for $q=0$.

The proof of Lemma~\ref{lem:0degreestep} shows that $p^*$ induces the identity on $HY_b^0(X\ast I,A)$. Namely if $\varphi\in \ker dY_0^b$ then there are only boundedly many $x\in X$ such that $\varphi$ has mixed values on $\{x\}\ast I$. Thus $\varphi$ is the same map as $(x,(s,t))\mapsto \varphi(x,(s+t,0))$ up to bounded error.
\end{proof}

 A metric space $X$ is called \emph{coarsely contractible} if the map $d(x_0,\cdot)$ is a coarse homotopy equivalence.

\begin{lem}
 If $X$ is a coarsely contractible metric space then $\check H_{ct}^q(X,A)=0$ for every $q>0$ and finite abelian group $A$.
\end{lem}
\begin{proof}
 By definition the map $d(x_0,\cdot)$ is a coarse homotopy equivalence. Therefore it induces an isomorphism in cohomology by Theorem~\ref{thm:homotopyinvariance}. Since $\Z_{\ge 0}$ is acyclic so is $X$ by Theorem~\ref{thm:hqztree}.
\end{proof}

\section{Cohomology of free abelian groups}
This chapter proves Theorem~\ref{thm:introZn}.

\begin{lem}
\label{lem:flasqueacyclic}
 If $X$ is a CAT$(0)$ space then $X\times\R_{\ge0}$ is coarsely contractible.
\end{lem}
\begin{proof}
 We define two maps
 \begin{align*}
\begin{aligned}
 \pi:X\times \R_{\ge0}&\to \R_{\ge0}\\
 (x,i)&\mapsto d(x,x_0)+i
 \end{aligned}
 &
 \begin{aligned}
 \iota:\R_{\ge 0}&\to X\times \R_{\ge0}\\
 i&\mapsto (x_0,i) 
 \end{aligned}
\end{align*}
and show that they are coarse homotopy inverses.

Since $X$ is CAT$(0)$ there exists for every $x\in X$ a geodesic $\gamma_x:[0,1]\to X$ joining $x_0$ to $x$. The inequality $d(\gamma_x(t),\gamma_y(t))\le d(x,y)$ holds for every $t\in[0,1]$ by the curvature condition.

We define 
\begin{align*}
 h:(X\times \R_{\ge0})\ast I&\to X\times \R_{\ge0}\\
 ((x,i),(s,t))&\mapsto (\gamma_x(\hat t),i+\hat s d(x,x_0)).
\end{align*}
Here $\hat s=\frac s{s+t}$ and $\hat t=\frac t{s+t}$. The map $h$ joins $\iota\circ \pi$ to $id_{X\times \Z_{\ge0}}$. It remains to show that $h$ is coarse. If $R\ge 0$ and $d(((x_1,i_1),(s_1,t_1)),((x_2,i_2),(s_2,t_2))\le R$ then 
\begin{align*}
 d(\gamma_{x_1}(\hat t_1),\gamma_{x_2}(\hat t_2))
 &\le d(\gamma_{x_1}(\hat t_1),\gamma_{x_1}(\hat t_2))+d(\gamma_{x_1}(\hat t_2),\gamma_{x_2}(\hat t_2))\\
 &\le |\hat t_1-\hat t_2|d(x_0,x_1)+d(x_1,x_2)\\
 &= \left|\frac{t_1}{s_1+t_1}-\frac{t_2}{s_2+t_2}\right|d(x_0,x_1)+d(x_1,x_2)\\
 &\le \left| \frac {t_1}{d(x_0,x_1)}-\frac{t_2}{d(x_0,x_1)-R}\right|d(x_0,x_1)+R\\
 &=\left|\frac{(d(x_0,x_1)-R)t_1-d(x_0,x_1)t_2}{d(x_0,x_1)-R}\right|+R\\
 &\le \frac{2R}{1-\frac R{d(x_0,x_1)}}+R\\
 &\le 4R+R
\end{align*}
for $d(x_0,x_1)$ large compared to $R$. Then
\begin{align*}
|i_1&+\hat s_1d(x_1,x_0)-(i_2+\hat s_2d(x_2,x_0))|\\
&\le |i_1-i_2|+|\hat s_1 d(x_1,x_0)-\hat s_1 d(x_2,x_0)|+|\hat s_1d(x_2,x_0)-\hat s_2 d(x_2,x_0)|\\
&\le R+ \hat s_1 d(x_1,x_2)+4R\\
&\le 6 R
\end{align*}
for $d(x_0,x_2)$ large compared to $R$. Thus $h$ is coarsely uniform.

If $S\ge 0$ and $((x,i),(s,t))\in h^{-1}(\Delta_S[(x_0,0)])$ then
\begin{align}
\label{eq:xt}
 \hat t d(x,x_0)=\hat t d(\gamma_x(1),\gamma_x(0))= d(\gamma_x(\hat t),\gamma_x(0))\le S
\end{align}
and 
\begin{align}
\label{eq:xs}
 \hat s d(x,x_0)\le S.
\end{align}
The inequalities~\ref{eq:xt} and~\ref{eq:xs} add to an inequality $d(x,x_0)\le S$. This inequality and $|i|\le S$ shows that $h$ is coarsely proper. This way we have showed that $h$ is a coarse map.

\end{proof}

Lemma~\ref{lem:flasqueacyclic} in particular implies that $\R_{\ge0}\times \R^{n-1}$ is a coarsely contractible subspace of $\R^n$. In fact $\R^i\times \R_{\ge 0}\times \R^{n-1-i}$ and $\R^i\times\R_{<0}\times \R^{n-1-i}$ are coarsely contractible subspaces of $\R^n$ and so is every finite intersection of them.

\begin{lem}
\label{lem:leraycover}
 If $\sheaff$ is a sheaf on a metric space $X$ and $(U_i)_i$ a Leray cover of $X$, namely a coarse cover such that every finite intersection $U_{i_0}\cap\cdots \cap U_{i_q}$ is $\sheaff$-acyclic, then
 \[
  \check H_{ct}^q(X,\sheaff)=\check H^q((U_i)_i,\sheaff).
 \]
The right side denotes \v Cech-cohomology of the cover $(U_i)_i$. 
\end{lem}
\begin{proof}
For sheaves on a topological space there exist a number of proofs for this result. We mimic the proof of \cite[Theorem~III.4.5]{Hartshorne1977}.

 Embed $\sheaff$ in a flabby sheaf $\sheafg$ an take the quotient $\sheaff_1$. Then there is a short exact sequence of sheaves
 \begin{align}
 \label{eq:sessheaves}
  0\to\sheaff\to \sheafg\to \sheaff_1\to 0.
 \end{align}
 Since $\check H^1(U_{i_0}\cap\cdots\cap U_{i_q},\sheaff)=0$ there is a short exact sequence of abelian groups
 \[
  0\to\sheaff(U_{i_0}\cap\cdots\cap U_{i_q})\to \sheafg(U_{i_0}\cap\cdots\cap U_{i_q})\to \sheaff_1(U_{i_0}\cap\cdots\cap U_{i_q})\to 0.
 \]
Taking products we obtain an exact sequence of \v Cech-cocomplexes
\[
  0\to\check C^*((U_i)_i,\sheaff)\to\check C^*((U_i)_i,\sheafg)\to \check C^*((U_i)_i,\sheaff_1)\to 0.
\]
This results in a long exact sequence of \v Cech cohomology. Since $\mathcal G$ is flabby, its \v Cech cohomology vanishes for $q>0$. This way we get an exact sequence
\begin{align}
\label{eq:les1}
 0\to \check H^0((U_i)_i,\sheaff)\to \check H^0((U_i)_i,\sheafg)\to \check H^0((U_i)_i,\sheaff_1)\to \check H^1((U_i)_i,\sheaff)\to 0.
\end{align}
and isomorphisms
\begin{align}
\label{eq:iso}
 \check H^q((U_i)_i,\sheaff_1)=\check H^{q+1}((U_i)_i,\sheaff)
\end{align}
for each $q\ge 1$. Associated to the exact sequence of sheaves~\ref{eq:sessheaves} there is an exact sequence
\begin{align}
\label{eq:les2}
 0\to \check H^0_{ct}(X,\sheaff)\to \check H^0_{ct}(X,\sheafg)\to \check H^0_{ct}(X,\sheaff_1)\to \check H^1_{ct}(X,\sheaff)\to 0.
\end{align}
Since $\check H^0((U_i)_i,\sheafh)=\check H
_{ct}^0(X,\sheafh)$ for any sheaf $\sheafh$ we can compare the exact sequences~\ref{eq:les1} and~\ref{eq:les2} and obtain
\[
 \check H^1((U_i)_i,\sheaff)=\check H^1_{ct}(X,\sheaff).
\]
Now the long exact sequence in cohomology for \ref{eq:sessheaves} and $U_{i_0}\cap \cdots\cap U_{i_q}$ being $\mathcal F$-acyclic implies that $U_{i_0}\cap \cdots\cap U_{i_q}$ is $\sheaff_1$-acyclic. Thus $\sheaff_1$ satisfies the conditions of this Lemma. This way we use induction and isomorphisms~\ref{eq:iso} to obtain the result for $q>1$.
\end{proof}

\begin{thm}
  We can compute cohomology:
 \[
  \check H^q_{ct}(\Z^n;A)=\begin{cases}
                      A\oplus A & n=1,q=0\\
                      A &n\not=1,q=0\vee q=n-1\\
                      0 &\mbox{otherwise}.
                     \end{cases}
 \]
 if $A$ is a finite abelian group.
\end{thm}
\begin{proof}
If $n=1$ this result is already Theorem~\ref{thm:hqztree}.

Suppose $n\ge 2$. We compute cohomology of $\R^n$. The result for $\Z^n$ follows since the spaces $\Z^n$ and $\R^n$ are coarsely equivalent. For $i=1,\ldots,n$ define $U^+_{i}:=\R^{i-1}\times \R_{\ge0}\times \R^{n-i}$ and $U^-_{i}:=\R^{i-1}\times \R_{<0}\times \R^{n-i}$. A finite intersection of those halfspaces is coarsely contractible by Lemma~\ref{lem:flasqueacyclic}. We show the $(U^+_{i},U^-_{i})_i$ form a coarse cover. Let $\bar x:=(x_1,\ldots,x_n)\in \bigcap_{i=1}^n(\Delta_R[(U^+_{i})^c]\cap \Delta_R[(U^-_{i})^c])$ be a point. We show $d(\bar x,(0,\ldots,0))\le nR$. If $i=1,\ldots,n$ then $\bar x\in \Delta_R[(U^-_{i})^c]$ implies $x_i\ge -R$. And $\bar x\in \Delta_R[(U^+_{i})^c]$ implies $x_i\le R$. Together they imply $d(x_i,0)\le R$ and in all together the result. Thus $\mathcal U:=(U_i^+,U_i^-)_i$ forms a Leray cover.

By Lemma~\ref{lem:leraycover} the metric cohomology of $\R^n$ is the \v Cech cohomology of $\mathcal U$. In the topological world the sphere $S^{n-1}$ admits a Leray cover by $V_i^+:=\{(x_1,\ldots,x_n)\in S^{n-1}|x_i>0\}$ and $V_i^-:=\{(x_1,\ldots,x_n)\in S^{n-1}|x_i<0\}$. The combinatorical information of this cover is the same as that of $\mathcal U$. Thus both covers have the same nerve. Since the nerve contains all the cohomological information we just proved that $\R^n$ (as a metric space) has the same cohomology as $S^{n-1}$ (as a topological space). This proves the claim.
\end{proof}

There is another method we can use to compute the cohomology of $\R^n$. If $X$ is a CAT(0) metric space then the \emph{coarse cone over $X$} is given by $X\times \R_{\ge0}$. Lemma~\ref{lem:flasqueacyclic} tells us that the coarse cone is coarsely contractible. Now the \emph{coarse suspension} of a CAT(0) metric space $X$ is given by $X\times \R$.

\begin{lem}
 If $X$ is a CAT(0) metric space then the coarse suspension shifts coarse sheaf cohomology by one degree, namely $\check H_{ct}^q(X,A)\cong\check H_{ct}^{q+1}(X\times \R,A)$ for $q\ge 1$.
 
\end{lem}
\begin{proof}
 We cover $X\times \R$ by two sets $U_1=\{(x,t)|t\le d(x_0,x)\}$ and $U_2=\{(x,t)|t\ge-d(x,x_0)\}$. They form a coarse cover: if $(y,s)\in U_1^c,(x,t)\in U_2^c$ with $d((y,s),(x,t))\le R$ then $|s-t|\le R$, $s>d(x_0,y)$ and $t<-d(x_0,x)$. Since $s$ is positive and $t$ is negative we obtain $|s|,|t|\le R$. Furthermore $d(x_0,y)<s\le R$ and $d(x_0,x)\le |t|\le R$. Thus $U_1,U_2$ coarsely cover $X$.
 
 There are coarse homotopy equivalences $U_1,U_2\simeq X\times \R_{\ge0}$ and $U_1\cap U_2\simeq X$. Namely the inclusion $i_1:X\times \R_{\ge0}\to U_2$ has a coarse homotopy inverse 
 \begin{align*}
  p_1:U_2&\to X\times \R_{\ge0}\\
  (x,t)&\mapsto\begin{cases}
                (x,t) & t\ge0\\
                (x,0) & t<0
               \end{cases}
 \end{align*}
The coarse homotopy connecting $i_1\circ p_1$ with $id_{U_2}$ is given by
\begin{align*}
 h_1:U_2\ast I&\to U_2\\
 ((x,t),(i,j))&\mapsto \begin{cases}
                        (x,\hat i t) & t<0\\
                        (x,t) & t\ge0.
                       \end{cases}
\end{align*}
Here $\hat i:=\frac i {i+j}$. We show $h$ is coarse: If $d(((x,t),(i,j)),((y,s),(n,m)))\le R$ then in particular $d(x,y)\le R$ and 
\[
d(\hat it,\hat ns)\le d(\hat it,\hat is)+d(\hat is,\hat ns)\le 2R.
\]
Thus $h_1$ is coarsely uniform. If $d((x,\hat it),(x_0,0))\le S$ then $d(x,x_0)\le S$ and if $t<0$ then $|t|\le d(x,x_0)\le S$ or if $t\ge0$ then $d(t,0)\le S$. Thus $h_1$ is coarsely proper.

The coarse homotopy equivalence connecting $X$ with $U_1\cap U_2$ is given by $i_2:x\mapsto (x,0)$ and its inverse is
\begin{align*}
 p_2:U_1\cap U_2&\to X\\
 (x,t)&\mapsto x.
\end{align*}
The coarse homotopy joining $i_2\circ p_2$ to $id_{U_1\cap U_2}$ is given by
\begin{align*}
 h_2:(U_1\cap U_2)\ast I &\to U_1\cap U_2\\
 ((x,t),(i,j))&\mapsto (x,\hat i t).
\end{align*}
We prove $h_2$ is coarsely proper, the property coarsely uniform can be shown similarly as for $h_1$. If $d((x,\hat it),(x_0,0))\le S$ then $d(x,x_0)\le S$ and $|t|\le d(x,x_0)\le S$.  
 
 Then the long exact sequence of Theorem~\ref{thm:mv} gives us
 \[
  \check H^q_{ct}(X,A)\cong \check H_{ct}^q(U_1\cap U_2,A)\cong \check H_{ct}^{q+1}(U_1\cup U_2,A)=\check H_{ct}^{q+1}(X\times \R,A)
 \]
in degree $q\ge 1$.
\end{proof}

\begin{rem}
 The suspension functor has a natural adjoint, the loop space. If $X$ is a metric space then the \emph{loop space of $X$}, $\Omega X$ consists of coarse maps $\R\to X$. A subset $E\s \Omega X\times \Omega X$ is an entourage if for every $R\ge 0$ the set $\{(\varphi(r),\psi(r'))\mid (\varphi,\psi)\in E,|r-r'|\le R\}$ is an entourage in $Y$. Note that $\Omega Y$ defined this way does not have a connected coarse structure. Then there exists a natural isomorphism 
 \[
  Hom(X\times \R,Y)=Hom(X,\Omega Y).
 \]
 Suppose coarse maps $\R^n\to Y$ denote the $n-1$th coarse homotopy group of $Y$. If we insert $\R^n$ for $X$ in the adjoint relation then we can see that the loop space shifts coarse homotopy groups down a dimension.
\end{rem}

\bibliographystyle{halpha-abbrv}
\bibliography{mybib}

\address

\end{document}